\documentclass[11pt,reqno]{amsart}

\usepackage{amssymb,amscd}
\newtheorem{thm}{Theorem}[section]
\newtheorem{prop}[thm]{Proposition}
\newtheorem{lemma}[thm]{Lemma}
\newtheorem{cor}[thm]{Corollary}
\newtheorem{definition}[thm]{Definition}
\newtheorem{remark}[thm]{Remark}
\newtheorem{example}[thm]{Example}

\numberwithin{equation}{section}

\def\cH{\mathcal{H}}
\def\bM{\mathbb{M}}
\def\bN{\mathbb{N}}
\def\bR{\mathbb{R}}
\def\bC{\mathbb{C}}
\def\ffi{\varphi}
\def\diag{\mathrm{Diag}}
\def\cP{\mathcal{P}}
\def\cF{\mathcal{F}}
\def\eps{\varepsilon}
\def\<{\langle}
\def\>{\rangle}

\begin{document}
\baselineskip=16pt

\title[Operator $k$-tone functions]
{Higher order extension of L\"owner's theory: Operator $k$-tone functions}

\author[U. FRanz]{Uwe Franz}
\address{D\'epartement de math\'ematiques de Besan\c con,
Universit\'e de Franche-Comt\'e,
16, route de Gray, 25 030 Besan\c con Cedex, France}
\email{uwe.franz@univ-fcomte.fr}
\author[F. Hiai]{Fumio Hiai}
\address{Graduate School of Information Sciences, Tohoku University,
Aoba-ku, Sendai 980-8579, Japan}
\email{hiai.fumio@gmail.com}
\author[\'E. Ricard]{\'Eric Ricard}
\address{Laboratoire de Math\'ematiques Nicolas Oresme,
Universit\'e de Caen-Basse Normandie,
BP 5186, 14032 Caen Cedex, France}
\email{eric.ricard@unicaen.fr}



\maketitle

\begin{abstract}
The new notion of operator/matrix $k$-tone functions is introduced, which is a higher
order extension of operator/matrix monotone and convex functions. Differential properties
of matrix $k$-tone functions are shown. Characterizations, properties, and examples of
operator $k$-tone functions are presented. In particular, integral representations of
operator $k$-tone functions are given, generalizing familiar representations of operator
monotone and convex functions.
\end{abstract}

\section*{Introduction}

The theory of operator/matrix monotone functions was initiated by the celebrated paper of
L\"owner \cite{Lo}, which was soon followed by Kraus \cite{Kr} on operator/matrix convex
functions. After further developments due to some authors (for instance, Bendat and Sherman
\cite{BS} and Koranyi \cite{Kor}), in their seminal paper \cite{HP} Hansen and Pedersen
established a modern treatment of operator monotone and convex functions. In
\cite{Do,An,Bh} (also \cite{Hi}) are found comprehensive expositions on the subject matter.
A remarkable feature of L\"owner's theory is that we have several characterizations of
operator monotone and convex functions from several different points of view. The
importance of complex analysis in studying operator monotone functions is well understood
from their characterization in terms of analytic continuation as Pick functions. Integral
representations for operator monotone and convex functions are essential ingredients of
the theory from both theoretical and application sides. The notion of divided differences has played a vital role in the theory from its
very beginning. In this connection, the operator/matrix-valued differential calculus, as
represented by Daleckii and Krein's derivative formula \cite{DK}, is quite useful as is
clearly mentioned in the survey paper \cite{Da}. Also, differential methods were adopted
in \cite{An2,HT1,HT2} for some analysis of matrix convex functions. 

The differential calculus of the form
\begin{equation}\label{eq-0.2}
{d^k\over dt^k}\,f(A+tX)\Big|_{t=0}
\end{equation}
is quite relevant to operator/matrix monotone and convex functions, where $A$ is a
self-adjoint operator with spectrum in the domain interval $(a,b)$ of $f$ and $X$ is a
positive operator. The above $k$th derivative exists for matrices whenever $f$ is $C^k$ on
$(a,b)$ (see \cite{BV,Hi}). In the infinite-dimensional setting, the existence of the
$k$th derivative is rather subtle but it exists in the operator norm when $f$ is analytic
(see \cite{Pe} for the $k$th derivative under a weaker assumption). It is well known
\cite{BV} that the operator/matrix monotonicity of $f$ is characterized by the positivity
of derivative \eqref{eq-0.2} for $k=1$, and the operator/matrix convexity is similarly
characterized by the positivity of \eqref{eq-0.2} for $k=2$. Our motivation for the present
paper came from the naive question of what is a higher order extension of the operator
monotonicity related to the higher order derivative given in \eqref{eq-0.2} for $k>2$.
Although the idea seems very natural, this kind of higher order extension of L\"owner's
theory is new while some other types of extensions have been discussed (see \cite{ABDR}
for instance).

In Section \ref{sec-1} of the paper we first define, for a real
function on $(a,b)$, operator/matrix-valued divided differences of $f$
by generalizing the usual divided differences. By using these
generalized divided differences we introduce, for $k,n\in\bN$, the
notions of operator $k$-tone functions and matrix $k$-tone functions
of order $n$, which are higher order extensions of operator/matrix
monotone and convex functions.  In fact, when $k=1$ and $k=2$,
operator $k$-tone functions are operator monotone and convex
functions, respectively. In the last of Section \ref{sec-1}, slightly
refined forms of the standard integral representations of operator
monotone functions on $(-1,1)$ and on $(0,\infty)$ are provided for
later use.  The main theorem of Section \ref{sec-2} gives
differentiability properties of matrix $k$-tone functions of order
$n$; such a function must be of class $C^{2n+2k-4}$. In particular, an
$n$-convex function is $C^{2n-2}$, extending the classical result of
Kraus \cite{Kr} when $n=2$. Section \ref{sec-3} presents
several characterizations of operator $k$-tone functions on $(a,b)$,
e.g., in terms of derivatives or divided differences of $f$. It
turns out that $f$ is operator $k$-tone on $(a,b)$ if and only if
for any $\alpha\in(a,b)$ there exists an operator monotone function
$g$ on $(a,b)$ such that
\begin{equation}\label{eq-0.3}
f(x)=\sum_{l=0}^{k-2}{f^{(l)}(\alpha)\over l!}\,(x-\alpha)^l+(x-\alpha)^{k-1}g(x),
\qquad x\in(a,b),
\end{equation}
that is, the $(k-1)$st remainder term of the Taylor series of $f$ at $\alpha$ is
given as $(x-\alpha)^{k-1}$ times an operator monotone function. This shows that operator
$k$-tone functions have rather simple structure with only additive
and multiplicative polynomial factors beyond operator monotone functions.
Sections \ref{sec-4} and \ref{sec-5} contain further properties of operator
$k$-tone functions on $(-1,1)$ and on $(0,\infty)$, respectively. In particular, we present
integral expressions of such functions, generalizing the well-known versions for operator
monotone functions (when $k=1$). Furthermore, in the last parts of Sections \ref{sec-4}
and \ref{sec-5}, we clarify what are the operator versions of absolutely monotone and
completely monotone functions on $(-1,1)$ and on $(0,\infty)$. Examples of operator
$k$-tone functions on $(0,\infty)$ are provided in Section \ref{sec-6}.
Finally, it is worth noting that the operator $k$-tonicity condition is weaker and weaker
as $k$ is bigger and bigger (see Propositions \ref{P-3.9} and \ref{P-5.2} for precise
statements), unlike the usual differentiability property of numerical functions.

\section{Definitions and preliminaries}\label{sec-1}
\subsection{Notations}
For each $n\in\bN$, $\bM_n$ is the $n\times n$ matrix algebra, $\bM_n^{sa}$ the set of
$n\times n$ Hermitian matrices, and $\bM_n^+$ the set of $n\times n$ positive semidefinite
matrices. Throughout the paper $\cH$ is a fixed infinite-dimensional separable Hilbert
space, $B(\cH)$ is the set of all bounded operators on $\cH$, $B(\cH)^{sa}$ the set of
self-adjoint operators on $\cH$, and $B(\cH)^+$ the set of positive operators on $\cH$.
The symbol $I$ denotes the identity matrix or the identity operator. For an open interval
$(a,b)$ in the real line $\bR$, we write $\bM_n^{sa}(a,b)$ and $B(\cH)^{sa}(a,b)$ for the
sets of all elements of $\bM_n^{sa}$ and of $B(\cH)^{sa}$, respectively, with spectra in
$(a,b)$, which are convex and open (in the norm topology) in $\bM_n^{sa}$ and
$B(\cH)^{sa}$, respectively. In particular, $\bM_n^{sa}(0,\infty)$ is the set of
$n\times n$ invertible positive semidefinite matrices. When $f$ is a real function on
$(a,b)$ and $A\in\bM_n^{sa}(a,b)$, $f(A)$ is the usual functional calculus of $A$ by $f$,
and $f(A)$ for $A\in B(\cH)^{sa}(a,b)$ is similar when $f$ is continuous.

\subsection{Operator/matrix-valued divided differences}
Let $f$ be a real function on an open interval $(a,b)$, where $-\infty\le a<b\le\infty$.
For distinct $x_0,x_1,x_2,\dots$ in $(a,b)$, the {\it divided differences} of $f$ are
recursively defined as
$$
f^{[1]}(x_0,x_1):={f(x_0)-f(x_1)\over x_0-x_1}
$$
and for $k=2,3,\dots$,
$$
f^{[k]}(x_0,x_1,\dots,x_k)
:={f^{[k-1]}(x_0,x_1,\dots,x_{k-1})-f^{[k-1]}(x_1,\dots,x_{k-1},x_k)\over x_0-x_k}.
$$
For each $k\in\bN$ the $k$th divided difference $f^{[k]}(x_0,x_1,\dots,x_k)$ can be
extended by continuity to a continuous function on $(a,b)^{k+1}$ whenever $f$ is $C^k$ on
$(a,b)$. See \cite[pp.\ 1--7]{Do} and \cite[Sect.\ 2.2]{Hi} for properties of divided
differences.

To introduce the key notion of operator or matrix $k$-tone functions, we need to extend
the above divided differences to the operator-valued or matrix-valued version. Of course
it does not make sense to replace the real variables $x_0,x_1,\dots$ in the above with
self-adjoint operators or matrices. To extend the divided differences to operators, we fix
two $A,B$ in $B(\cH)^{sa}(a,b)$ or in $\bM_n^{sa}(a,b)$ and distinct $t_0,t_1,t_2,\dots$
in $[0,1]$. Let $X_k:=(1-t_k)A+t_kB$ for $k=0,1,2,\dots$, and define the
{\it operator-valued} or {\it matrix-valued divided difference} of $f$ as follows:
$$
f^{[1]}(A,B;t_0,t_1):={f(X_0)-f(X_1)\over t_0-t_1}
$$
and for $k=2,3,\dots$,
$$
f^{[k]}(A,B;t_0,t_1,\dots,t_k)
:={f^{[k-1]}(A,B;t_0,t_1,\dots,t_{k-1})
-f^{[k-1]}(A,B;t_1,\dots,t_{k-1},t_k)\over t_0-t_k}.
$$

In particular, for $\alpha,\beta\in(a,b)$ with $\alpha<\beta$ and for distinct
$t_0,t_1,\dots\in[0,1]$ let $x_k:=(1-t_k)\alpha+t_k\beta$; we then notice that
$$
f^{[k]}(x_0,x_1,\dots,x_k)I
={1\over(\beta-\alpha)^k}\,f^{[k]}(\alpha I,\beta I;t_0,t_1,\dots,t_k),
$$
from which we can consider $f^{[k]}(A,B;t_0,t_1,\dots,t_k)$ as a natural operator
or matrix version of the usual $k$th divided difference.
(A further generalization of divided difference for functions on vector spaces was
recently proposed in \cite{Be}.)

\begin{lemma}\label{L-1.1}
Let $A,B$ and $t_k$, $X_k$ for $k=0,1,\dots$ be as above. For every $k\in\bN$,
\begin{align*}
f^{[k]}(A,B;t_0,t_1,\dots,t_k)
&=\sum_{l=0}^k{f(X_l)\over\prod_{0\le j\le k,\,j\ne l}(t_l-t_j)} \\
&={\sum_{l=0}^k(-1)^{k-l}\prod_{0\le i<j\le k,\,i,j\ne l}(t_j-t_i)f(X_l)
\over\prod_{0\le i<j\le k}(t_j-t_i)}.
\end{align*}
Hence, $f^{[k]}(A,B;t_0,t_1,\dots,t_k)$ is symmetric in the variables $t_0,t_1,\dots,t_k$.
\end{lemma}

\begin{proof}
The first equality is easy to prove by induction on $k$ and the second is a simple
rewriting.
\end{proof}

\begin{example}\label{E-1.2}\rm
Let $m\in\bN$, $f(x)=x^m$, and let $A,B$ be in $B(\cH)^{sa}$ or in $\bM_n^{sa}$. Using an
induction on $k$ one can easily verify that
\begin{align*}
&f^{[k]}(A,B;t_0,t_1,\dots,t_k) \\
&\quad=F_{k,m-k}(B-A,A)
+\sum_{l=k+1}^m\left(\sum_{j_0,j_1,\dots,j_k\ge0\atop j_0+j_1+\dots+j_k=l-k}
t_0^{j_0}t_1^{j_1}\cdots t_k^{j_k}\right)F_{l,m-l}(B-A,A),
\end{align*}
for every $k\in\bN$, where $F_{l,m-l}(X,Y)$ denotes the sum of all products of
$l$ $X$'s and $m-l$ $Y$'s for $X,Y\in B(\cH)^{sa}$ (this is defined to be zero unless
$0\le l\le m$). In particular,
\begin{align*}
&f^{[m]}(A,B;t_0,t_1,\dots,t_m)=F_{m,0}(B-A,A)=(B-A)^m, \\
&f^{[k]}(A,B;t_0,t_1,\dots,t_k)=0\qquad\mbox{for all $k>m$}.
\end{align*}
\end{example}

It is known \cite[Theorem 2.1]{BV} (also \cite[Theorem 2.3.1]{Hi}) that if $f$ is $C^k$ on
$(a,b)$, then the matrix functional calculus $f(A)$ is $k$ times Fr\'echet differentiable
at every $A\in\bM_n^{sa}(a,b)$ and the $k$th Fr\'echet derivative $D^kf(A)$, a
multi-linear map from $(\bM_n^{sa})^k$ into $\bM_n^{sa}$, is continuous in $A$.
Consequently, the $k$th derivative
$$
{d^k\over dt^k}\,f(A+tX)\Big|_{t=0}=D^kf(A)(\underbrace{X,\dots,X}_k)
$$
exists for every $A\in\bM_n^{sa}(a,b)$ and $X\in\bM_n^{sa}$ and is continuous in $A$ and
$X$. For infinite-dimensional Hilbert space operators, it was shown in \cite{DK} that
$t\mapsto f(A+tX)$ is differentiable and expressed as a double operator integral under
the $C^2$ assumption of $f$, and later Birman and Solomyak developed the general theory of
double operator integrals (its concise account is found in \cite{HK}). However, the
situation in the infinite-dimensional case is rather subtle; indeed the $C^1$ assumption
of $f$ is not sufficient for the differentiability of $f(A+tX)$ as mentioned in \cite{Pe}.
Taking this into account, we restrict ourselves to the matrix-valued case for the following
continuous extendability property of $f^{[k]}(A,B;t_0,t_1,\dots,t_n)$.

\begin{prop}\label{P-1.3}
Assume that $f$ is $C^k$ on $(a,b)$. Then for every $n\in\bN$ and every
$A,B\in\bM_n^{sa}(a,b)$, $f^{[k]}(A,B;t_0,t_1,\dots,t_k)$ for distinct
$t_0,t_1,\dots,t_k\in[0,1]$ can be extended by continuity to a function on the whole
$[0,1]^{k+1}$ so that $f^{[k]}(A,B;t_0,t_1,\dots,t_k)$ is continuous in
$(A,B;t_0,t_1,\dots,t_k)\in(\bM_n^{sa}(a,b))^2\times[0,1]^{k+1}$.
\end{prop}

\begin{proof}
Let $A,B\in\bM_n^{sa}(a,b)$ and choose a $\delta>0$ so that $(1-t)A+tB\in\bM_n^{sa}(a,b)$
for all $t\in(-\delta,1+\delta)$. For any state $\omega$ on $\bM_n$ define
$$
\phi_\omega(t):=\omega(f((1-t)A+tB)),\qquad t\in(-\delta,1+\delta).
$$
For any distinct $t_0,t_1,\dots,t_k$ in $[0,1]$ it is obvious by definition that
\begin{equation}\label{eq-1.1}
\phi_\omega^{[k]}(t_0,t_1,\dots,t_k)=\omega(f^{[k]}(A,B;t_0,t_1,\dots,t_k)).
\end{equation}
From the $C^k$ assumption of $f$, the matrix-valued function $t\mapsto f((1-t)A+tB)$ is
$C^k$ on $(-\delta,1+\delta)$ as remarked before the proposition. Hence $\phi_\omega$ is
$C^k$ on $(-\delta,1+\delta)$ so that $\phi_\omega^{[k]}(t_0,t_1,\dots,t_k)$ can extend to
a continuous function on $(-\delta,1+\delta)^{k+1}$. Since this is the case for every
state $\omega$, we see that $f^{[k]}(A,B;t_0,t_1,\dots,t_k)$ extends to a continuous
function on $[0,1]^{k+1}$. Furthermore, the stronger continuity of $f^{[k]}$ in the whole
variables $(A,B;t_0,t_1,\dots,t_k)$ can be shown in a way similar to the proof of
\cite[Lemma 2.2.4]{Hi} by using the continuity of ${d^l\over dt^l}\,f((1-t)A+tB)$,
$0\le l\le k$, in $(A,B,t)$. We omit the details for this.
\end{proof}

\subsection{Definition of operator/matrix $k$-tone functions}
\begin{definition}\label{D-1.4}\rm
Let $k\in\bN$ and $f$ be a real continuous function on $(a,b)$. We say that $f$ is
{\it operator $k$-tone} on $(a,b)$ if, for every $A,B\in B(\cH)^{sa}(a,b)$ with $A\le B$
and for any $0=t_0<t_1<\dots<t_k=1$, we have
\begin{equation}\label{eq-1.2}
f^{[k]}(A,B;t_0,t_1,\dots,t_k)\ge0,
\end{equation}
or equivalently, due to Lemma \ref{L-1.1},
\begin{equation}\label{eq-1.3}
\sum_{l=0}^k(-1)^{k-l}\prod_{0\le i<j\le k\atop i,j\ne l}(t_j-t_i)f(X_l)\ge0,
\end{equation}
where $X_l:=(1-t_l)A+t_lB$, $0\le l\le k$. Moreover, for each $n\in\bN$, if a real (not
necessarily continuous) function $f$ on $(a,b)$ satisfies \eqref{eq-1.2}, or equivalently
\eqref{eq-1.3}, for every $A,B\in\bM_n^{sa}(a,b)$ with $A\le B$ and for any
$t_0,t_1,\dots,t_k$ as above, then we say that $f$ is {\it matrix $k$-tone of order $n$}
on $(a,b)$. A matrix $k$-tone function of order $1$ (i.e., a $k$-tone function in the
numerical sense) is said to be {\it $k$-tone} for short.
\end{definition}

When $k=1$, inequality \eqref{eq-1.3} is nothing but $-f(A)+f(B)\ge0$ so that $f$ is
operator $1$-tone on $(a,b)$ if and only if it is operator monotone on $(a,b)$. When $k=2$,
\eqref{eq-1.3} is
$$
(1-t_1)f(A)-f((1-t_1)A+t_1B)+t_1f(B)\ge0,\qquad0<t_1<1.
$$
Hence $f$ is matrix $2$-tone of order $n$ on $(a,b)$ if and only if $f$ is conditionally
$n$-convex there (the conditional matrix convexity here means the matrix convexity under
condition $A\le B$). Note \cite{Kr} (also \cite[Theorem 2.4.4]{Hi}) that the conditional
$n$-convexity is equivalent to the usual $n$-convexity. Thus, Definition \ref{D-1.4}
may be considered as a natural higher order extension of operator/matrix monotone or convex
functions.

By Example \ref{E-1.2} note that any polynomial function $\sum_{l=0}^k\alpha_kx^k$ of real
coefficients with $\alpha_k\ge0$ is operator $k$-tone on the whole $\bR$.

\begin{lemma}\label{L-1.5}
Let $k,n\in\bN$. A real function $f$ on $(a,b)$ is matrix $k$-tone of order $n$ if and
only if $f$ satisfies \eqref{eq-1.2} for every $A,B\in\bM_n^{sa}(a,b)$ and for any distinct
$t_0,t_1,\dots,t_k\in[0,1]$. The same is true for the operator $k$-tonicity of a real
continuous function $f$.
\end{lemma}

The lemma is easily shown by using Lemma \ref{L-1.1}, and the following
proposition is shown by a standard convergence argument (as will also be used
in the proof of Corollary \ref{C-3.5}\,(c)). Proofs of these may be omitted here.

\begin{prop}\label{P-1.6}
Let $k\in\bN$ and $f$ be a real continuous function on $(a,b)$. Then $f$ is operator
$k$-tone if and only if it is matrix $k$-tone of every order $n$.
\end{prop}

\subsection{Integral representations of operator monotone functions}

In this subsection we show integral representations of operator monotone functions on
$(-1,1)$ and on $(0,\infty)$, which will be useful in our later discussions. Although such
integral representations are well known as described in \cite[Sect.\ V.4]{Bh} (also
\cite[Sect.\ 2.7]{Hi}), our representations below are slight modifications of the standard
ones, including a new insight on certain universality of the representing measure. The
representations will indeed be extended to operator $k$-tone functions in Theorems
\ref{T-4.1} and \ref{T-5.1}.

\begin{thm}\label{T-1.8}
Let $f$ be an operator monotone function on $(-1,1)$. Then there exists a unique finite
positive measure $\mu$ on $[-1,1]$ such that for any choice of $\alpha\in(-1,1)$,
\begin{equation}\label{eq-1.5}
f(x)=f(\alpha)+\int_{[-1,1]}{x-\alpha\over(1-\lambda x)(1-\lambda\alpha)}\,d\mu(\lambda),
\qquad x\in(-1,1).
\end{equation}
\end{thm}

\begin{proof}
For each fixed $\alpha\in(-1,1)$ consider the transformation $\ffi_\alpha$ from $[-1,1]$
onto itself defined by
$$
\ffi_\alpha(t):={t+\alpha\over1+\alpha t},\qquad t\in[-1,1],
$$
which is an operator monotone function with $\ffi_\alpha(0)=\alpha$. Apply \cite[V.4.5]{Bh}
to obtain a unique finite positive measure $m_\alpha$ on $[-1,1]$ such that
\begin{equation}\label{eq-1.6}
(f\circ\ffi_\alpha)(t)=f(\alpha)+\int_{[-1,1]}{t\over1-\kappa t}\,dm_\alpha(\kappa),
\qquad t\in(-1,1).
\end{equation}
Defining a finite positive measure $\mu_\alpha$ on $[-1,1]$ by
$$
d\mu_\alpha(\lambda):={1-\lambda\alpha\over1+\alpha\ffi_\alpha^{-1}(\lambda)}
\,dm_\alpha(\ffi_\alpha^{-1}(\lambda)),
$$
we have
\begin{equation}\label{eq-1.7}
f(x)=f(\alpha)+\int_{[-1,1]}{x-\alpha\over(1-\lambda x)(1-\lambda\alpha)}
\,d\mu_\alpha(\lambda),\qquad x\in(-1,1),
\end{equation}
which is representation \eqref{eq-1.5} while $\mu=\mu_\alpha$ is depending on
$\alpha\in(-1,1)$ at the moment. Moreover, since expression \eqref{eq-1.7} can conversely
be converted into \eqref{eq-1.6}, it is seen that a representing measure $\mu_\alpha$ in
\eqref{eq-1.7} is unique.

Now, we prove that $\mu_\alpha$ is independent of the parameter $\alpha\in(-1,1)$.
For any $\alpha,\beta\in(-1,1)$, inserting
$$
{x-\alpha\over(1-\lambda x)(1-\lambda\alpha)}
={\beta-\alpha\over(1-\lambda\alpha)(1-\lambda\beta)}
+{x-\beta\over(1-\lambda x)(1-\lambda\beta)}
$$
into \eqref{eq-1.7} we have
$$
f(x)=f(\alpha)
+\int_{[-1,1]}{\beta-\alpha\over(1-\lambda\alpha)(1-\lambda\beta)}\,d\mu_\alpha(\lambda)
+\int_{[-1,1]}{x-\beta\over(1-\lambda x)(1-\lambda\beta)}\,d\mu_\alpha(\lambda).
$$
Letting $x=\beta$ gives
$$
f(\beta)=f(\alpha)
+\int_{[-1,1]}{\beta-\alpha\over(1-\lambda\alpha)(1-\lambda\beta)}\,d\mu_\alpha(\lambda),
$$
and $\mu_\alpha=\mu_\beta$ follows from the uniqueness of $\mu_\beta$ representing $f$ in
\eqref{eq-1.7} with $\beta$ in place of $\alpha$, so the theorem has been proved.
\end{proof}

Note that from \eqref{eq-1.5} we have
$$
f'(\alpha)=\int_{[-1,1]}{1\over(1-\lambda\alpha)^2}\,d\mu(\lambda),\qquad\alpha\in(-1,1).
$$
In particular, $\mu([-1,1])=f'(0)$.

The theorem has the following corollary, which will play an essential role to prove the
main theorem of Section \ref{sec-3}.

\begin{cor}\label{C-1.9}
Let $f$ be an operator monotone function on $(-1,1)$ with the representing
measure $\mu$ as in Theorem \ref{T-1.8}. For every $\alpha\in (-1,1)$ and every
$m,k\in\bN$ with $m\ge k$,
$$
\bigl((x-\alpha)^{k-1}f\bigr)^{[m]}(x_1,x_2,\dots,x_{m+1})
=\int_{[-1,1]}{\lambda^{m-k}(1-\lambda\alpha)^{k-1}\over
(1-\lambda x_1)(1-\lambda x_2)\cdots(1-\lambda x_{m+1})}\,d\mu(\lambda)
$$
for all $x_1,x_2,\dots,x_{m+1}\in(-1,1)$, where $\lambda^{m-k}\equiv1$ on $[-1,1]$ if
$m=k$.
\end{cor}
\begin{proof}
For every $\lambda\in[-1,1]\setminus\{0\}$ we have
\begin{align}
{(x-\alpha)^k\over(1-\lambda x)(1-\lambda\alpha)^k}
&={\{(1-\lambda\alpha)-(1-\lambda x)\}^k\over\lambda^k(1-\lambda x)(1-\lambda\alpha)^k}
\nonumber\\
&={1\over\lambda^k(1-\lambda x)}+(\mbox{a polynomial of degree $k-1$}).\label{eq-1.a}
\end{align}
Since $m\ge k$, we hence have
\begin{align*}
&\biggl({(x-\alpha)^k\over(1-\lambda x)(1-\lambda\alpha)^k}\biggr)^{[m]}
(x_1,x_2,\dots,x_{m+1}) \\
&\qquad={1\over\lambda^k}\biggl({1\over1-\lambda x}\biggr)^{[m]}(x_1,x_2,\dots,x_{m+1})
={\lambda^{m-k}\over(1-\lambda x_1)(1-\lambda x_2)\cdots(1-\lambda x_{m+1})}
\end{align*}
for all $x_1,x_2,\dots,x_{k+1}\in(-1,1)$. The above certainly holds for $\lambda=0$ as
well. Therefore,
integrating against the measure $(1-\lambda \alpha)^{k-1}\,d\mu$
gives the result as we can take the $k$th divided
difference inside the integral in \eqref{eq-1.5}.
\end{proof}

\begin{thm}\label{T-1.10}
Let $f$ be an operator monotone function on $(0,\infty)$. Then there exists a unique
$\gamma\ge0$ and a unique positive measure $\mu$ on $[0,\infty)$ such that
$$
\int_{[0,\infty)}{1\over(1+\lambda)^2}\,d\mu(\lambda)<+\infty
$$
and for any choice of $\alpha\in(0,\infty)$,
\begin{equation}\label{eq-1.8}
f(x)=f(\alpha)+\gamma(x-\alpha)
+\int_{[0,\infty)}{x-\alpha\over(x+\lambda)(\alpha+\lambda)}\,d\mu(\lambda),
\qquad x\in(0,\infty).
\end{equation}
\end{thm}

\begin{proof}
For each $\alpha\in(0,\infty)$ consider the transformation $\psi_\alpha$ from
$[-1,1)$ onto $[0,\infty)$ defined by
$$
\psi_\alpha(t):={\alpha(1+t)\over1-t},\qquad t\in[-1,1),
$$
which is operator monotone on $[-1,1)$. Representing $f\circ\psi_\alpha$ as in
\eqref{eq-1.6} and defining $\gamma_\alpha\ge0$ and a positive measure $\mu_\alpha$ on
$[0,\infty)$ by
$$
\gamma_\alpha:={m_\alpha(\{1\})\over2\alpha},\qquad
d\mu_\alpha(\lambda):={\alpha+\lambda\over1-\psi_\alpha^{-1}(\lambda)}
\,dm_\alpha(\psi_\alpha^{-1}(\lambda)),
$$
we have the integral expression of $f$. The proof remaining is similar to that of
Theorem \ref{T-1.8}, so the details are omitted.
\end{proof}

From \eqref{eq-1.8} we have
$$
f'(\alpha)=\gamma+\int_{[0,\infty)}{1\over(\alpha+\lambda)^2}\,d\mu(\lambda),
\qquad\alpha\in(0,\infty),
$$
and hence $\gamma=\lim_{\alpha\to\infty}f'(\alpha)$.

\section{Differentiability properties of matrix $k$-tone functions}\label{sec-2}

It is well known \cite{Lo,Do} that if $f$ is $n$-monotone (or
matrix $1$-tone in our terminology) on $(a,b)$, then it
is $C^{2n-3}$ on $(a,b)$ and $f^{(2n-3)}$ is convex there. Also, a
primary result of \cite{Kr} is that if $f$ is conditionally $2$-convex
(or matrix $2$-tone), then it is $C^2$ there. The main aim of this section is
to prove the next theorem extending the above results to matrix $k$-tone functions
of order $n$ for general $k$ and $n$. In particular, when $k=2$, the theorem shows
differentiability results for $n$-convex functions. It seems that assertion
(a) is new even in this particular case where $k=2$ and $n>2$. Results in \cite{Lo,Kr}
say that when $k=1$ and $k=2$ property (c) is not only necessary but also sufficient for
$f$ to be matrix $k$-tone of order $n$. Also, see \cite{Do,HT1,HT2} for property (d)
when $k=1,2$.

\begin{thm}\label{T-2.1}
Let $k,n\in\bN$ and assume that a real function $f$ on $(a,b)$ is matrix $k$-tone of order
$n$. Then the following properties {\rm(a)}--{\rm(d)} hold:
\begin{itemize}
\item[\rm(a)] $f$ is $C^{2n+k-4}$ on $(a,b)$ if $k\ge2$ or $n\ge2$.
\item[\rm(b)] The following functions are convex on $(a,b)$:
$$
\begin{cases}
f',f^{(3)},\dots,f^{(2n-3)} & \text{if $k=1$ and $n\ge2$}, \\
f^{(k-2)},f^{(k)},\dots,f^{(2n+k-4)} & \text{if $k\ge2$ and $n\ge 1$}.
\end{cases}
$$
\item[\rm(c)] The matrix
$$
\Biggl[f^{[k]}(x_i,x_j,\underbrace{x_1,\dots,x_1}_{k-1})\Biggr]_{i,j=1}^n
$$
is positive semidefinite for any choice of $x_1,\dots,x_n$ from $(a,b)$ if $n\ge2$.
\item[\rm(d)] The matrix
$$
\Biggl[{f^{(i+j+k)}(x)\over(i+j+k)!}\Biggr]_{i,j=0}^{n-1}
$$
exists and is positive semidefinite for almost every $x\in(a,b)$.
\end{itemize}
\end{thm}

To prove the theorem we start with

\begin{lemma}\label{L-2.2}
Let $f$ be a $C^k$ function on $(a,b)$ satisfying the assumption of Theorem
\ref{T-2.1}. Then
$$
{d^k\over dt^k}\,f(A+tX)\Big|_{t=0}\ge0
$$
for every $A\in\bM_n^{sa}(a,b)$ and every $X\in\bM_n^+$.
\end{lemma}

\begin{proof}
Let $A\in\bM_n^{sa}(a,b)$ and $X\in\bM_n^+$. We may assume that $A+X\in\bM_n^{sa}(a,b)$;
then a $\delta>0$ is chosen so that $A+tX\in\bM_n^{sa}(a,b)$ for
all $t\in(-\delta,1)$. For any state $\omega$ on $\bM_n$ define
$$
\phi_\omega(t):=\omega(f(A+tX)),\qquad t\in(-\delta,1),
$$
which is $C^k$ on $(-\delta,1)$ due to the $C^k$ assumption on $f$. When
$0=t_0<t_1<\dots<t_k<1$, we have as \eqref{eq-1.1},
$$
\phi_\omega^{[k]}(t_0,t_1,\dots,t_k)=\omega(f^{[k]}(A,A+X;t_0,t_1,\dots,t_k))\ge0
$$
by Lemma \ref{L-1.5}. Letting $t_l\searrow0$ for $1\le l\le k$ gives
$$
0\le\phi_\omega^{[k]}(0,0,\dots,0)={1\over k!}\,\phi_\omega^{(k)}(0)
={1\over k!}\,\omega\biggl({d^k\over dt^k}\,f(A+tX)\Big|_{t=0}\biggr)
$$
by \cite[p.\ 6]{Do} (also \cite[Lemma 2.2.4]{Hi}). The conclusion follows since the state
$\omega$ is arbitrary.
\end{proof}

We next prove properties (c) and (d) of the theorem under the additional assumption of $f$
being $C^\infty$.

\begin{lemma}\label{L-2.3}
Let $f$ be a $C^\infty$ function on $(a,b)$ satisfying the assumption of Theorem
\ref{T-2.1}. Then {\em(c)} holds including the case $n=1$ and {\rm(d)} holds
for every $x\in(a,b)$.
\end{lemma}

\begin{proof}
(c)\enspace
For every $x_1,\dots,x_n\in(a,b)$ let $A:=\diag(x_1,\dots,x_n)$, the diagonal matrix with
diagonal entries $x_1,\dots,x_n$. According to Daleckii and Krein's derivative formula in
the matrix case (see \cite[Theorem 2.3.1]{Hi}), for every $X\in\bM_n^{sa}$ we have
\begin{equation}\label{eq-2.a}
{d^k\over dt^k}\,f(A+tX)\bigg|_{t=0}
=\Biggl[\sum_{r_1,\dots,r_{k-1}=1}^nk!f^{[k]}(x_i,x_{r_1},\dots,x_{r_{k-1}},x_j)
X_{ir_1}X_{r_1r_2}\cdots X_{r_{k-1}j}\Biggr]_{i,j=1}^n.
\end{equation}
For any $\xi_1,\dots,\xi_n\in\bC$ let
$X:=\bigl[\overline\xi_i\xi_j\bigr]_{i,j=1}^n\in\bM_n^+$. Lemma \ref{L-2.2} then 
implies that for any $\zeta_1,\dots,\zeta_n\in\bC$,
$$
\sum_{i,j=1}^n\,\sum_{r_1,\dots,r_{k-1}=1}^nf^{[k]}(x_i,x_{r_1},\dots,x_{r_{k-1}},x_j)
\overline\xi_i|\xi_{r_1}|^2\cdots|\xi_{r_{k-1}}|^2\xi_j\overline\zeta_i\zeta_j\ge0.
$$
We may replace above $\zeta_i$ with $\zeta_i/\xi_i$ under the assumption that $\xi_i\ne0$
for all $i$. Now let $\xi_1=1$ and $\xi_r\to0$ for $r\ne1$ to obtain
$$
\sum_{i,j=1}^nf^{[k]}(x_i,\underbrace{x_1,\dots,x_1}_{k-1},x_j)
\overline\zeta_i\zeta_j\ge0.
$$

(d)\enspace
For any fixed $x\in(a,b)$ define a $C^\infty$ function $g$ on $(a,b)$ by
$$
g(t):=f^{[k-1]}(t,x,\dots,x),\qquad t\in(a,b).
$$
It is plain to notice that
\begin{equation}\label{eq-2.1}
g(t)={1\over(t-x)^{k-1}}\Biggl\{
f(t)-\sum_{l=0}^{k-2}{f^{(l)}(x)\over l!}(t-x)^l\Biggr\},\qquad t\in(a,b),
\end{equation}
where $g(t)=f(t)$ for $k=1$. Set $\delta:=\min\{x-a,b-x\}/n$. For every $h\in(0,\delta)$
we then have
$$
G_h:=\bigl[g^{[1]}(x+ih,x+jh)\bigr]_{i,j=0}^{n-1}
=\bigl[f^{[k]}(x+ih,x+jh,x,\dots,x)\bigr]_{i,j=0}^{n-1}\ge0
$$
thanks to (c) proved above. By Taylor's theorem we expand $g^{[1]}(x+ih,x+jh)$ as
$$
g^{[1]}(x+ih,x+jh)=\sum_{m=0}^{2n-2}{g^{(m+1)}(x)\over(m+1)!}
\cdot{i^{m+1}-j^{m+1}\over i-j}\,h^m+o(h^{2n-2})
$$
with convention $(i^{m+1}-j^{m+1})/(i-j)=m+1$ if $i=j$. Therefore,
$$
G_h=\sum_{m=0}^{2n-2}{g^{(m+1)}(x)\over(m+1)!}
\Biggl(\sum_{l=0}^mu_l\otimes u_{m-l}\Biggr)h^m+o(h^{2n-2}),
$$
where $u_l:=\bigl(0^l,1^l,\dots,(n-1)^l\bigr)\in\bC^n$ and
$u_l\otimes u_{m-l}:=\bigl[i^lj^{m-l}\bigr]_{i,j=0}^{n-1}$ for $0\le l\le m\le2n-2$ (with
$0^0:=1$). For every $\zeta_0,\dots,\zeta_{n-1}\in\bC$, since $u_0,\dots,u_{n-1}$ are
linearly independent, there exists a $v\in\bC^n$ such that $\<u_l,v\>=\zeta_lh^{-l}$ for
$l=0,\dots,n-1$. Since $\<u_l,v\>=O\bigl(h^{-(n-1)}\bigr)$ if $l\ge n$, one can easily
verify that
$$
\<v,(u_l\otimes u_{m-l})v\>h^m=\overline{\<u_l,v\>}\<u_{m-l},v\>h^m=O(h)
$$
if $l\ge n$ or $m-l\ge n$. Therefore,
\begin{align*}
0\le\<v,G_hv\>&=\sum_{m=0}^{2n-2}{g^{(m+1)}(x)\over(m+1)!}
\Biggl(\sum_{0\le l\le n-1,\,0\le m-l\le n-1}\overline{\zeta_l}\zeta_{m-l}\Biggr)+O(h) \\
&=\sum_{i,j=0}^{n-1}{g^{(i+j+1)}(x)\over(i+j+1)!}\,\overline{\zeta_i}\zeta_j+O(h),
\end{align*}
and letting $h\searrow0$ yields that
$\bigl[g^{(i+j+1)}(x)/(i+j+1)!\bigr]_{i,j=0}^{n-1}\ge0$. It immediately follows from
\eqref{eq-2.1} that
$$
g^{(l+1)}(x)={l!f^{(l+k)}(x)\over(l+k)!},\qquad l=0,1,\dots
$$
and we have the conclusion.
\end{proof}

\begin{lemma}\label{L-2.4}
Let $k\in\bN$ with $k\ge2$ and assume that $f$ is $k$-tone on $(a,b)$. Then $f$ is
$C^{k-2}$ and $f^{(k-2)}$ is convex on $(a,b)$. That is, {\rm(a)} and {\rm(b)} of Theorem
\ref{T-2.1} hold when $n=1$.
\end{lemma}

\begin{proof}
The proof is by induction on $k$. The case $k=2$ is obvious since 2-tonicity means
convexity. Assume that $f$ is $(k+1)$-tone on $(a,b)$. For any $c\in(a,b)$, since
$f^{[1]}(x,c)$ is $k$-tone on $(a,c)$, it follows from induction hypothesis that
$f^{[1]}(x,c)$ is $C^{k-2}$ on $(a,c)$. Hence $f$ is $C^{k-2}$ on $(a,b)$. If
$x_1,\dots,x_{k+1},y_1,\dots,y_{k+1}$ are distinct in $(a,b)$ and $x_i<y_i$ for all
$i=1,\dots,k+1$, then the $(k+1)$-tonicity of $f$ implies that
$f^{[k]}(x_1,\dots,x_{k+1})\le f^{[k]}(y_1,\dots,y_{k+1})$. For any $a'<b'$ in
$(a,b)$ choose $\alpha_1<\dots<\alpha_{k+1}$ in $(a,a')$ and
$\beta_1<\dots<\beta_{k+1}$ in $(b',b)$. The above inequality then implies that
$$
f^{[k]}(\alpha_1,\dots,\alpha_{k+1})\le f^{[k]}(x_1,\dots,x_{k+1})
\le f^{[k]}(\beta_1,\dots,\beta_{k+1})
$$
for every distinct $x_1,\dots,x_{k+1}\in(a',b')$. Hence there exists a $K>0$ (depending on
$a',b'$) such that $\big|f^{[k]}(x_1,\dots,x_{k+1})\big|\le K$ for all distinct
$x_1,\dots,x_{k+1}\in(a',b')$. This in turn implies that if $x_1,\dots,x_k,y_1,\dots,y_k$
are distinct in $(\alpha,\beta)$, then
\begin{equation}\label{eq-2.2}
\big|f^{[k-1]}(x_1,\dots,x_k)-f^{[k-1]}(y_1,\dots,y_k)\big|\le K\sum_{i=1}^k|x_i-y_i|.
\end{equation}
For every $\alpha,\beta,x,y\in(a',b')$ such that $x\ne\alpha$ and $y\ne\beta$, let
$x_1\to x$, $y_1\to y$, $x_2,\dots,x_k\to\alpha$, and $y_1,\dots,y_k\to\beta$ in
\eqref{eq-2.2} to obtain
\begin{equation}\label{eq-2.3}
\big|f^{[k-1]}(x,\alpha,\dots,\alpha)-f^{[k-1]}(y,\beta,\dots,\beta)\big|
\le K\{|x-y|+(k-1)|\alpha-\beta|\},
\end{equation}
where $f^{[k-1]}(x,\alpha,\dots,\alpha)$ is well defined as
$$
f^{[k-1]}(x,\alpha,\dots,\alpha)={1\over(x-\alpha)^{k-1}}
\Biggl\{f(x)-\sum_{l=0}^{k-2}{f^{(l)}(\alpha)\over l!}\,(x-\alpha)^l\Biggr\}
$$
due to the $C^{k-2}$ of $f$. By \eqref{eq-2.3} with $\beta=\alpha$ we have the limit
$$
\theta(\alpha):=\lim_{x\to\alpha}f^{[k-1]}(x,\alpha,\dots,\alpha),
\qquad\alpha\in(a',b').
$$
For every $\alpha,\beta\in(a',b')$, letting $x\to\alpha$ and $y\to\beta$ in \eqref{eq-2.3},
we have $|\theta(\alpha)-\theta(\beta)|\le Kk|\alpha-\beta|$ so that $\theta$ is continuous
on $(a',b')$. Furthermore, by \eqref{eq-2.3} we have
$$
\big|f^{[k-1]}(x,\alpha,\dots,\alpha)-\theta(\alpha)\big|\le K|x-\alpha|,
\qquad x,\alpha\in(a',b'),\ x\ne\alpha.
$$
Letting $r(x,\alpha):=f^{[k-1]}(x,\alpha,\dots,\alpha)-\theta(\alpha)$ with
$r(\alpha,\alpha)=0$ we write
\begin{equation}\label{eq-2.4}
f(x)=\sum_{l=0}^{k-2}{f^{(l)}(\alpha)\over l!}\,(x-\alpha)^l
+\theta(\alpha)(x-\alpha)^{k-1}+r(x,\alpha)(x-\alpha)^{k-1},
\quad x,\alpha\in(a',b').
\end{equation}
Since $f^{(l)}(\alpha)$ for $0\le l\le k-2$ and $\theta(\alpha)$ are continuous in
$\alpha\in(a',b')$ and $|r(x,\alpha)|\le K|x-\alpha|$ for $x,\alpha\in(a',b')$,
expression \eqref{eq-2.4} shows \cite[pp.\ 6--9]{AR} (also \cite[Lemma A.1.1]{Hi}) that
$f$ is $C^{k-1}$ on $(a',b')$ with $f^{(k-1)}(\alpha)=(k-1)!\,\theta(\alpha)$. The
$C^{k-1}$ of $f$ on $(a,b)$ follows since $a',b'$ are arbitrary.

To complete the induction procedure, it remains to prove that $f^{(k-1)}$ is convex on
$(a,b)$. To do so, we adopt a standard regularization technique (see \cite[pp.\ 11--13]{Do}
for example). Let $\phi(t)$ be a non-negative $C^\infty$ function on $\bR$ supported in
$[-1,1]$ such that $\int\phi(s)\,ds=1$. For every $\eps>0$ sufficiently small define
$$
f_\eps(x):={1\over\eps}\int_{x-\eps}^{x+\eps}\phi\biggl({x-s\over\eps}\biggr)f(s)\,ds
=\int_{-1}^1\phi(s)f(x-\eps s)\,ds,\qquad x\in(a+\eps,b-\eps),
$$
which is a $C^\infty$ function on $(a+\eps,b-\eps)$. For every distinct $x_1,\dots,x_{k+2}$
in $(a+\eps,b-\eps)$ one can easily see that
$$
f_\eps^{[k+1]}(x_1,\dots,x_{k+2})
=\int_{-1}^1\phi(s)f^{[k+1]}(x_1-\eps s,\dots,x_{k+2}-\eps s)\,ds\ge0.
$$
Hence $f_\eps^{(k+1)}(x)\ge0$ for all $x\in(a+\eps,b-\eps)$ so that $f_\eps^{(k-1)}$ is
convex on $(a+\eps,b-\eps)$. Since $f$ is $C^{k-1}$ on $(a,b)$ as already proved,
$f_\eps^{(k-1)}(x)\to f^{(k-1)}(x)$ as $\eps\searrow0$ for all $x\in(a,b)$ and the
convexity of $f^{(k-1)}$ follows.
\end{proof}

\begin{lemma}\label{L-2.5}
Let $\{f_m\}$ be a sequence of $C^1$ functions on $[a,b]$ such that the finite limits
$\lim_{m\to\infty}f_m(a)$ and $\lim_{m\to\infty}f_m(b)$ exist. Assume that $f_m'$ is convex
on $[a,b]$ for every $m$ and there exists a $K>0$ such that $f_m'(x)\le K$ for all
$x\in[a,b]$ and all $m$. Then $\{f_m'\}$ is uniformly bounded on $[a,b]$ and hence
$\{f_m\}$ is uniformly equicontinuous on $[a,b]$.
\end{lemma}

\begin{proof}
We can assume that $K=0$. Since $f_m'\le0$ is convex on $[a,b]$, we have
$$
f_m(b)-f_m(a)=\int_a^bf_m'(t)\,dt\le\biggl({b-a\over2}\biggr)f_m'(x)
$$
for all $x\in[a,b]$. Hence $\{f_m'\}$ is also uniformly bounded below on $[a,b]$, and so
$\{f_m\}$ is uniformly equicontinuous there.
\end{proof}

We are now in a position to complete the proof of the theorem.

\bigskip\noindent
{\it Proof of Theorem \ref{T-2.1}.}\enspace
The theorem when $k=1$ and $n\ge2$ is L\"owner's result \cite[p.\ 76]{Do}, so we may assume
that $k\ge2$. Let us prove (a) and (b) for all $k\ge2$ by induction on $n$. The initial
case $n=1$ is Lemma \ref{L-2.4}. Let $n\in\bN$ and assume that $f$ is matrix $k$-tone of
order $n+1$ on $(a,b)$. Since $f$ is matrix $k$-tone of order $n$, it follows from
induction hypothesis that $f$ is $C^{2n+k-4}$ and $f^{(k-2)}$, $f^k$, $\dots$,
$f^{(2n+k-4)}$ are convex on $(a,b)$. Define regularizations $f_\eps$ of $f$ for small
$\eps>0$ as in the last part of the proof of Lemma \ref{L-2.4}. For every
$A,B\in\bM_{n+1}^{sa}(a+\eps,b-\eps)$ with $A\le B$ one can easily see that if
$0=\lambda_0<\lambda_1<\dots<\lambda_k=1$ then
$$
f_\eps^{[k]}(A,B;\lambda_0,\dots,\lambda_k)
=\int_{-1}^1\phi(s)f^{[k]}(A-\eps sI,B-\eps sI;\lambda_0,\dots,\lambda_k)\,ds\ge0.
$$
This means that $f_\eps$ is matrix $k$-tone of order $n+1$ on $(a+\eps,b-\eps)$. So one
can apply (d) of Lemma \ref{L-2.3} to $f_\eps$ with $n+1$ in place of $n$
to see that $f_\eps^{(2n+k)}(x)\ge0$ for all $x\in(a+\eps,b-\eps)$ and hence
$f_\eps^{(2n+k-2)}$ is convex on $(a+\eps,b-\eps)$. Since $f^{(2n+k-4)}$ is convex on
$(a,b)$ as already mentioned, $f^{(2n+k-3)}(x)$ exists and hence
$f_\eps^{(2n+k-3)}(x)\to f^{(2n+k-3)}(x)$ as $\eps\searrow0$ for all $x\in(a,b)$ except
at most countable points. Choose $a'<b'$ in $(a,b)$, arbitrarily near $a,b$ respectively,
at which $f^{(2n+k-4)}$ is differentiable. Since $\bigl\{f_\eps^{(2n+k-4)}\bigr\}$
(for $\eps>0$ sufficiently small) is uniformly bounded above on $[a',b']$, we see by
Lemma \ref{L-2.5} that $\bigl\{f_{\eps}^{(2n+k-3)}\bigr\}$ (for small $\eps>0$) is
uniformly equicontinuous on $[a',b']$. Hence there exists a continuous function $\ffi$ on
$[a',b']$ such that $f_{\eps}^{(2n+k-3)}(x)\to\ffi(x)$ uniformly on $[a',b']$. Note that
$f_{\eps}^{(2n+k-4)}(x)\to f^{(2n+k-4)}(x)$ for all $x\in(a,b)$. Hence $f^{(2n+k-4)}$ is
differentiable with $f^{(2n+k-3)}(x)=\ffi(x)$ on $(a',b')$, and therefore, $f$ is
$C^{2n+k-3}$ on $(a,b)$.

Since $f^{(2n+k-4)}$ is convex on $(a,b)$, $f^{(2n+k-3)}$ is non-decreasing on $(a,b)$ and
so differentiable almost everywhere on $(a,b)$ by Lebesgue's theorem. Let $D$ be the set
of $x\in(a,b)$ at which $f^{(2n+k-3)}$ is differentiable. Note that
$f_\eps^{(2n+k-3)}(x)\to f^{(2n+k-3)}(x)$ as $\eps\searrow0$ for all $x\in(a,b)$ and
$f_\eps^{(2n+k-2)}(x)\to f^{(2n+k-2)}(x)$ for all $x\in D$. Choose
$\alpha_1<\alpha_2<a'<b'<\beta_1<\beta_2$ from $D$. Since $f_\eps^{(2n+k-2)}$ is convex on
$(a+\eps,b-\eps)$, we have
$$
{f_{\eps}^{(2n+k-2)}(\alpha_1)-f_{\eps}^{(2n+k-2)}(\alpha_2)\over\alpha_1-\alpha_2}
\le f_{\eps}^{(2n+k-1)}(x)\le
{f_{\eps}^{(2n+k-2)}(\beta_1)-f_{\eps}^{(2n+k-2)}(\beta_2)\over\beta_1-\beta_2}
$$
for all $x\in[a',b']$, which implies that
$\bigl\{f_{\eps}^{(2n+k-2)}\bigr\}$ is uniformly equicontinuous on $[a',b']$. Hence
there exists a continuous function $\psi$ on $[a',b']$ such that
$f_{\eps}^{(2n+k-2)}(x)\to\psi(x)$ uniformly on $[a',b']$. This shows that $f^{(2n+k-3)}$
is differentiable with $f^{(2n+k-2)}(x)=\psi(x)$ on $(a',b')$. Therefore, $f$ is
$C^{2n+k-2}$ on $(a,b)$. Moreover, $f^{(2n+k-2)}$ is convex on $(a,b)$ since
$f_\eps^{(2n+k-2)}(x)\to f^{(2n+k-2)}(x)$ for all $x\in(a,b)$.

Next, we prove (c) and (d). Let $f$ be matrix $k$-tone of order $n$ on $(a,b)$ and
$f_\eps$ be the regularization of $f$ as above. Then by Lemma \ref{L-2.3}, (c) and (d)
hold for $f_\eps$ on $(a+\eps,b-\eps)$. When $n\ge2$, since $k\le2n+k-4$, $f$ is
$C^k$ on $(a,b)$ by assertion (a), so (c) for $f$ follows by taking the limit of (c)
for $f_\eps$. Since $f^{(2n+k-4)}$ is convex on $(a,b)$, the same argument as above
shows that $f^{(2n+k-2)}(x)$ exists for almost every $x\in(a,b)$ and that
$f_\eps^{(2n+k-3)}(x)\to f^{(2n+k-3)}(x)$ and $f_\eps^{(2n+k-2)}(x)\to f^{(2n+k-2)}(x)$
for almost every $x\in(a,b)$. Hence (d) for $f$ follows as the almost everywhere limit
of (d) for $f_\eps$.\qed

\bigskip
As immediately seen from \cite[p.\ 6]{Do}, a real $C^k$ function on $(a,b)$ is $k$-tone
if and only if $f^{(k)}(x)\ge0$ for all $x\in(a,b)$. This can be extended to
the differential calculus in $n\times n$ matrices as follows: 

\begin{prop}\label{P-2.6}
Let $k,n\in\bN$ with $n\ge2$ and $f$ be a real function on $(a,b)$. Then the following
conditions are equivalent:
\begin{itemize}
\item[\rm(i)] $f$ is matrix $k$-tone of order $n$ on $(a,b)$;
\item[\rm(ii)] $f$ is $C^k$ on $(a,b)$ and
$$
{d^k\over dt^k}\,f(A+tX)\Big|_{t=0}\ge0
$$
for every $A\in\bM_n^{sa}(a,b)$ and every $X\in\bM_n^+$;
\item[\rm(iii)] $f$ is $C^k$ on $(a,b)$ and
$$
f(A+X)\ge\sum_{l=0}^{k-1}{1\over l!}\,D^lf(A)(\underbrace{X,\dots,X}_l)
$$
for every $A\in\bM_n^{sa}(a,b)$ and every $X\in\bM_n^+$ such that $A+X\in\bM_n^{sa}(a,b)$,
where $D^lf(A)$ is the $l$th Fr\'echet derivative of $A\mapsto f(A)$ on $\bM_n^{sa}(a,b)$.
\end{itemize}
\end{prop}

\begin{proof}
(i)\,$\Rightarrow$\,(ii).\enspace
Since $n\ge2$, assertion (i) implies the $C^k$ of $f$ due to Theorem \ref{T-2.1}\,(a),
and hence the implication was already proved in Lemma \ref{L-2.2}.

(ii)\,$\Rightarrow$\,(i).\enspace
Let $A,B\in\bM_n^{sa}(a,b)$. As in the proof of Proposition \ref{P-1.3}, choose a
$\delta>0$ and define $\phi_\omega$ for any state $\omega$ on $\bM_n$. For any
$0=t_0<t_1<\dots<t_k=1$, there exists a $\xi\in[0,1]$ such that
$$
\phi_\omega^{[k]}(t_0,t_1,\dots,t_k)={1\over k!}\,\phi_\omega^{(k)}(\xi)
={1\over k!}\,\omega\biggl({d^k\over dt^k}\,f(A+t(B-A))\Big|_{t=\xi}\biggr),
$$
which is non-negative due to assumption (ii). Hence by \eqref{eq-1.1},
$$
\omega(f^{[k]}(A,B;t_0,t_1,\dots,t_k))\ge0
$$
for all states $\omega$, so (i) follows.

(ii)\,$\Rightarrow$\,(iii).\enspace
Let $A,X$ be as stated in (iii); there is a $\delta>0$ such that $A+tX\in\bM_n^{sa}(a,b)$
for all $t\in(-\delta,1+\delta)$. For any state $\omega$ on $\bM_n$ define
$\phi_\omega(t):=\omega(f(A+tX))$ for $t\in(-\delta,1+\delta)$. By Taylor's theorem there
exists a $\theta\in(0,1)$ such that
\begin{equation}\label{eq-2.5}
\phi_\omega(1)=\sum_{l=0}^{k-1}{\phi_\omega^{(l)}(0)\over l!}
+{\phi_\omega^{(k)}(\theta)\over k!}.
\end{equation}
Notice that
$$
\phi_\omega^{(l)}(0)=\omega\biggl({d^l\over dt^l}\,f(A+tX)\Big|_{t=0}\biggr)
=\omega(D^lf(A)(X,\dots,X)),\qquad 0\le l\le k-1,
$$
and
$$
\phi_\omega^{(k)}(\theta)=\omega\biggl({d^k\over dt^k}\,f(A+tX)\Big|_{t=\theta}\biggr)
=\omega\biggl({d^k\over dt^k}\,f((A+\theta X)+tX)\Big|_{t=0}\biggr)\ge0
$$
due to (ii). Hence
$$
\omega\biggl(f(A+X)-\sum_{l=0}^{k-1}{1\over l!}\,d^lf(A)(X,\dots,X)\biggr)\ge0
$$
for all states $\omega$, so (iii) follows.

(iii)\,$\Rightarrow$\,(ii).\enspace
Let $A,X$ be as in (ii); we may assume that $A+X\in\bM_n^{sa}(a,b)$. For each $t\in(0,1)$,
as in \eqref{eq-2.5} there exists a $\theta_t\in(0,1)$ such that
$$
{1\over k!}\,\omega(D^kf(A+\theta_ttX)(X,\dots,X))
=\omega\Biggl(f(A+tX)-\sum_{l=0}^{k-1}{1\over l!}\,D^l(A)(tX,\dots,tX)\Biggr)\ge0.
$$
Since $D^kf(B)$ is continuous in $B\in\bM_n^{sa}(a,b)$, letting $t\searrow0$ we have
$\omega(D^kf(A)(X,\dots,X))\ge0$ for all states $\omega$. Hence
${d^k\over dt^k}\,f(A+tX)\big|_{t=0}=D^kf(A)(X,\dots,X)\ge0$.
\end{proof}

\section{Characterizations of operator $k$-tone functions}\label{sec-3}

The aim of this section is to present general characterizations of operator
$k$-tone functions on $(a,b)$. A well-known theorem of Bernstein is stated in
\cite[Chapter IV, Theorem 3a]{Wi} as follows: If $f$ is absolutely monotone (i.e., all
$f^{(k)}$, $k=0,1,2,\dots$, are non-negative) on $[a,b)$, then $f$ can be analytically
continued into $\{z\in\bC:|z-a|<b-a\}$. Its variation due to Valiron given in
\cite[pp.\ 160--161]{Co} says that if a function $f$ on $(a,b)$ has non-negative even
derivatives $f,f'',f^{(4)},\dots$, then it is analytic on $(a,b)$. We will need a little
bit more as given in the following:

\begin{lemma}\label{L-3.1} 
Let $f$ be operator $k$-tone on $(a,b)$, then $f$ is analytic on $(a,b)$ and
moreover the radius of convergence of the Taylor expansion of $f$ at $x\in (a,b)$ is 
$\delta_x:=\min\{x-a,b-x\}$.
\end{lemma}

\begin{proof}
By Propositions \ref{P-1.6}, \ref{P-2.6} and Theorem \ref{T-2.1}, we know that
$g:=f^{(k)}$ is $C^\infty$ on $(a,b)$ and moreover $g$, $g''$, $g^{(4)}$, $\dots$ are
non-negative. Take $x\in (a,b)$ and for $|h|<\delta_x$ let
$2\tilde g(h):=g(x+h)+g(x-h)$. Clearly $\tilde g$ is absolutely monotone on 
$(-\delta_x,\delta_x)$. By Bernstein's theorem mentioned above, $\tilde g$ is analytic
and its  Taylor expansion at $x$, which is
$\sum_{n=0}^\infty\bigl(g^{(2n)}(x)/2n!\bigr)h^{2n}$, has radius of convergence at least
$\delta_x$. On the other hand, by Theorem \ref{T-2.1}\,(d), for every $n\ge 0$, 
$$
\biggl(\frac{g^{(2n+1)}(x)}{(2n+1+k)!}\biggr)^2\leq
\frac{g^{(2n)}(x)}{(2n+k)!}\cdot \frac{g^{(2n+2)}(x)}{(2n+2+k)!},
$$
which shows that $\sum_{n=0}^\infty\bigl(g^{(n)}(x)/n!\bigr) h^{n}$ also has
radius of convergence at least $\delta_x$. From these estimates one can easily see,
using any Taylor formula, that
$g(x+h)=\sum_{n=0}^\infty\bigl(g^{(n)}(x)/n!\bigr) h^{n}$ for
$h\in (\delta_x/2,\delta_x/2)$. Thus $g$ is analytic. But it coincides
on $(x-\delta_x/2,x+\delta_x/2)$ with its Taylor expansion at $x$ which is known to have
radius of convergence at least $\delta_x$. Thus by the analytic continuation principle,
they have to agree on $(x-\delta_x,x+\delta_x)$. The conclusion on $f$ now follows.
\end{proof}

\begin{lemma}\label{L-3.2}
Let $g$ be an operator monotone function on $(a,b)$ and $\alpha\in (a,b)$, then 
$(x-\alpha)^{k-1}g$ is operator $(k+2l)$-tone for any $k\in\bN$ and any $l\in\bN\cup\{0\}$.
\end{lemma}

\begin{proof}
Set $m:=k+2l$. By Lemma \ref{L-3.1}, $f:=(x-\alpha)^{k-1}g$ is analytic.
By Propositions \ref{P-1.6} and \ref{P-2.6} we need to check that
${d^{m}\over dt^{m}}\,f(A+tX)\big|_{t=0}\geq0$ for all
$A\in\bM_n^{sa}(a,b)$ and $X\in\bM_n^+$ for any $n\in\bN$. Since this is a local estimate,
using a restriction, a translation and a dilation, we can always
assume that actually $f$ is operator monotone on $(-1,1)$, $\alpha\in(-1,1)$, and
the spectrum of $A$ sits in $(-1,1)$.

It suffices to assume that $A$ is diagonal so that $A=\diag(a_1,\dots,a_n)$. Using
Daleckii and Krein's derivative formula in \eqref{eq-2.a} and the formula for
divided differences given in Corollary \ref{C-1.9}, we have
\begin{align}
&{d^{m}\over dt^{m}}\,f(A+tX)\bigg|_{t=0} \nonumber\\
&\quad=m!\int_{[-1,1]}
D(\lambda)^{1/2}(D(\lambda)^{1/2}XD(\lambda)^{1/2})^mD(\lambda)^{1/2}
\lambda^{2l}(1-\lambda\alpha)^{k-1}\,d\mu(\lambda)\ge0, \label{eq-3.1}
\end{align}
where
$$
D(\lambda):=\diag\biggl({1\over1-\lambda a_1},\dots,{1\over1-\lambda a_n}\biggr),
\qquad\lambda\in[-1,1].
$$
\end{proof}

\begin{thm}\label{T-3.3}
Let $f$ be a real function on $(a,b)$, where $-\infty\le a<b\le\infty$. Let
$k\in\bN$. Then the following conditions {\rm(i)}--{\rm(vi)} are equivalent:
\begin{itemize}
\item[\rm(i)] $f$ is operator $k$-tone on $(a,b)$;
\item[\rm(ii)] $f$ is matrix $k$-tone of order $n$ on $(a,b)$ for every $n\in\bN$;
\item[\rm(iii)] $f$ is $C^k$ on $(a,b)$ and
$$
{d^k\over dt^k}\,f(A+tX)\Big|_{t=0}\ge0
$$
for every $A\in\bM_n^{sa}(a,b)$ and every $X\in\bM_n^+$ for any $n\in\bN$;
\item[\rm(iv)] $f$ is analytic on $(a,b)$ and
$$
{d^k\over dt^k}\,f(A+tX)\Big|_{t=0}\ge0
$$
for every $A\in B(\cH)^{sa}(a,b)$ and every $X\in B(\cH)^+$, where the
above derivative of order $k$ is well defined in the operator norm;
\item[\rm(v)] $f$ is $C^{k-2}$ on $(a,b)$ {\rm(}this is void for $k=1${\rm)} and
$f^{[k-1]}(x,\underbrace{\alpha,\cdots,\alpha}_{k-1})$ is operator monotone on $(a,b)$
for some {\rm(}equivalently, any{\rm)} $\alpha\in(a,b)$
{\rm(}with continuation of value at $x=\alpha${\rm)};
\item[\rm(vi)] $f$ is $C^{k-1}$ on $(a,b)$ and
$f^{[k-1]}(x,\alpha_1,\cdots,\alpha_{k-1})$ {\rm(}this is $f(x)$ for $k=1${\rm)} is
operator monotone on $(a,b)$ for some {\rm(}equivalently, any{\rm)} choice of
$\alpha_1,\dots,\alpha_n$ from $(a,b)$;
\item[\rm(vii)] $f$ is $C^\infty$ on $(a,b)$ and
$$
\Biggl[{f^{(i+j+k)}(x)\over(i+j+k)!}\Biggr]_{i,j=0}^{n-1}
$$
is positive semidefinite for every $x\in(a,b)$ and every $n\in\bN$.
\end{itemize}
\end{thm}

\begin{proof}
The equivalence of (i), (ii) and (iii) is included in Propositions \ref{P-1.6}
and \ref{P-2.6}.

(i)\,$\Rightarrow$\,(iv).\enspace By Lemma \ref{L-3.1}, $f$ is analytic. It only remains
to justify that $A\mapsto f(A)$ is $C^k$ on $B(\cH)^{sa}(a,b)$ and the inequality in (iv).
Let $A\in B(\cH)^{sa}(a,b)$ with spectrum included in $[x_0-h,x_0+h]\subset(a,b)$, $h>0$.
By Lemma \ref{L-3.1}, $f$ is equal to its Taylor expansion at $x_0$ on
$(x_0-\delta_{x_0},x_0+\delta_{x_0})$. Thus for any $B$ with spectrum in
$(x_0-\delta_{x_0},x_0+\delta_{x_0})$ we have
$f(B)=\sum_{m=0}^\infty c_m (B-x_0I)^m$ with $c_m:= f^{(m)}(x_0)/m!$, and so
$B\mapsto f(B)$ is $C^\infty$ on $B(\cH)^{sa}(x_0-\delta_{x_0},x_0+\delta_{x_0})$, a
neighborhood of $A$ as $h<\delta_{x_0}$. For every $X\in B(\cH)^+$ we have
\begin{equation}\label{eq-3.2}
{d^k\over dt^k}\,f(A+tX)\Big|_{t=0}
=\sum_{m=0}^\infty c_m{d^k\over dt^k}\,(A-x_0I+tX)^m\Big|_{t=0}
=\sum_{m=k}^\infty c_mF_{k,m-k}(X,A-x_0I),
\end{equation}
where $F_{k,m-k}(X,Y)$ was introduced in Example \ref{E-1.2}. On the other hand, for
$0=t_0<t_1<\dots<t_k<1$, by assumption (i) we have
$$
0\le f^{[k]}(A,A+X;t_0,t_1,\dots,t_k)
=\sum_{m=0}^\infty c_m\bigl(x^m\bigr)^{[k]}(A-x_0I,A-x_0I+X;t_0,t_1,\dots,t_k).
$$
Now apply the formula of Example \ref{E-1.2} to each term of the above expansion and then
let $t_i\searrow0$ for $1\le i\le k$ to obtain
$\sum_{m=k}^\infty c_mF_{k,m-k}(X,A-x_0I)\ge0$.

(iv)\,$\Rightarrow$\,(iii) is obvious, and so (i)--(iv) are equivalent.

Next, we prove that (i)\,$\Rightarrow$\,(v) for any $\alpha\in(a,b)$. Assume
(i); then $f$ is analytic. For every $n\in\bN$ and for any choice of
$\alpha,x_1,\dots,x_n$ from $(-1,1)$, apply (c) of Lemma \ref{L-2.3} to $n+1$ points
$\alpha,x_1,\dots,x_n$ and take the $n\times n$ submatrix deleting the first row and the
first column to obtain
$$
\biggl[{f^{[k-1]}(x_i,\alpha,\dots,\alpha)-f^{[k-1]}(x_j,\alpha,\dots,\alpha)
\over x_i-x_j}\biggr]_{i,j=1}^n
=\bigl[f^{[k]}(x_i,\alpha,\dots,\alpha,x_j)\bigr]_{i,j=1}^n\ge0,
$$
which gives (vi) due to L\"owner's theorem \cite{Lo} (or \cite[V.3.4]{Bh}).
(Note also that L\"owner's theorem easily follows from Proposition \ref{P-2.6}.)

Conversely, assume (v) for some $\alpha\in(a,b)$; so $f$ is $C^{k-2}$ and
$g(x):=f^{[k-1]}(x,\alpha,\dots,\alpha)$ is operator monotone on
$(a,b)$ (with continuation $g(\alpha)=\beta$ at $x=\alpha$). As in \eqref{eq-2.1} we have
\begin{equation}\label{eq-3.3}
f(x)=\sum_{l=0}^{k-2}{f^{(l)}(\alpha)\over l!}(x-\alpha)^l+(x-\alpha)^{k-1}g(x),
\qquad x\in(a,b),
\end{equation}
where $f(x)=g(x)$ for $k=1$. Hence Lemma \ref{L-3.2} yields (i).

Before going to (vi), let us prove that (i) implies that $(x-\alpha)f$ is
operator $(k+1)$-tone for any $\alpha\in(a,b)$. By (i)\,$\Rightarrow$\,(v) and
\eqref{eq-3.3} we have a polynomial $P$ of degree at most $k-2$ and an operator monotone
function $g$ so that $f=P+(x-\alpha)^{k-1}g$. Hence
$(x-\alpha)f=(x-\alpha)P+(x-\alpha)^kg$. By Lemma \ref{L-3.2} this yields the operator
$(k+1)$-tonicity of $(x-\alpha)f$.

Now, assume that (vi) holds for some choice of $\alpha_1,\dots,\alpha_n$ from $(a,b)$. 
As easily verified, $g(x)=f^{[k-1]}(x,\alpha_1,\dots,\alpha_{k-1})$ is of the form
$$
{f(x)-(\mbox{a polynomial of at most degree $k-2$})\over
(x-\alpha_1)(x-\alpha_2)\cdots(x-\alpha_{k-1})}
$$
so that 
\begin{equation}\label{eq-3.4}
f(x)=P(x)+\Biggl\{\prod_{l=1}^{k-1}(x-\alpha_l)\Biggr\}g(x),
\end{equation} 
where $P$ is a polynomial of degree at most $k-2$ and $g$ is operator monotone
on $(a,b)$. Hence (i) follows by applying, $k-1$ times, the result we have proved above.

Conversely, let us prove that (i)\,$\Rightarrow$\,(vi) for all possible
choices of $\alpha_1,\dots\alpha_{k-1}$. Assume (i), so $f$ is analytic on $(a,b)$. Let
$\alpha_1,\dots,\alpha_{k-1}$ be arbitrary in $(a,b)$. By (i)\,$\Rightarrow$\,(v),
$f^{[k-1]}(x,\alpha_1,\dots,\alpha_1)$ is operator monotone on $(a,b)$.
Let $g_1(x):=f^{[1]}(x,\alpha_1)$; then we have 
$g_1^{[k-2]}(x,\alpha_1,\dots,\alpha_1)=f^{[k-1]}(x,\alpha_1,\dots,\alpha_1)$.
Hence by (v)\,$\Rightarrow$\,(i) with some $\alpha$ and $k-1$ in place of $k$, $g_1$ is
operator $(k-1)$-tone on $(a,b)$ and hence $g_1^{[k-2]}(x,\alpha_2,\dots,\alpha_2)$ is
operator monotone on $(a,b)$. Repeat this argument to
$g_2(x):=g_1^{[1]}(x,\alpha_2)$, $\dots$, $g_{k-1}:=g_{k-2}^{[1]}(x,\alpha_{k-1})$ and
notice that $g_{k-1}(x)=f^{[k-1]}(x,\alpha_1,\dots,\alpha_{k-1})$. Hence (v) follows.

Theorem \ref{T-2.1} shows that (i)\,$\Rightarrow$\,(vii). Conversely, assume
(vii). By the same argument as in the proof of Lemma \ref{L-3.2} we can assume that
$(a,b)=(-1,1)$. The proof of Lemma 3.1 can be performed under assumption (vii) so that
the Taylor expansion $\sum_{n=0}^\infty\bigl(f^{(n)}(0)/n!\bigr)x^n$ is convergent on
$(-1,1)$. Thus we can write
$$
f(x)=\sum_{l=0}^{k-1}{f^{(l)}(0)\over l!}\,x^l+x^{k-1}g(x),\qquad
x\in(-1,1),
$$
where $g(x):=\sum_{n=0}^\infty\bigl(f^{(n+k)}(0)/(n+k)!\bigr)x^{n+1}$. Now, the same proof
as that of \cite[Theorem 2.8]{BS} appealing to the Hamburger moment problem can be done
due to (vii) to see that $g$ is operator monotone on $(-1,1)$. Hence Lemma \ref{L-3.2}
yields (i).
\end{proof}

In the rest of the section we will point out some interesting consequences of the above
theorem and its proof, which have mostly appeared from the connections between $f$ and
its divided differences in \eqref{eq-3.3} and \eqref{eq-3.4}.

\begin{cor}\label{C-3.4}
Let $f$ be a real function on $(a,b)$, where $-\infty\le a<b\le\infty$. Let $k\in\bN$.
Then the following conditions are equivalent:
\begin{itemize}
\item[\rm(i)] $f$ is operator $k$-tone on $(a,b)$;
\item[\rm(ii)] $f$ is $C^{k-2}$ on $(a,b)$ {\rm(}this is void for $k=1${\rm)} and
for some {\rm(}equivalently, any{\rm)} $\alpha\in(a,b)$ there exists an operator monotone
function $g$ on $(a,b)$ such that
$$
f(x)=\sum_{l=0}^{k-2}{f^{(l)}(\alpha)\over l!}\,(x-\alpha)^l+(x-\alpha)^{k-1}g(x);
$$
\item[\rm(iii)] for some {\rm(}equivalently, any{\rm)} $\alpha_1,\dots,\alpha_{k-1}$ in
$(a,b)$ there exist a polynomial $P(x)$ of degree less than or equal to $k-2$ and an
operator monotone function $g$ on $(a,b)$ such that
$$
f(x)=P(x)+\Biggl\{\prod_{l=1}^{k-1}(x-\alpha_l)\Biggr\}g(x).
$$
\end{itemize}
\end{cor}

The corollary tells us that the structure of operator $k$-tone functions is rather simple
with additive and multiplicative polynomial factors beyond operator monotone functions.
With $\alpha\in(a,b)$ fixed in (iii), there is a one-to-one correspondence, up to
polynomials of degree less than or equal to $k-2$, between operator $k$-tone functions on
$(a,b)$ and operator monotone functions on $(a,b)$.

It is also worthwhile to observe

\begin{cor}\label{C-3.5}
Let $f$ be an operator $k$-tone function on $(a,b)$, where
$-\infty\le a<b\le\infty$. Then 
\begin{itemize}
\item[\rm(a)] For any $\alpha\in (a,b)$, $(x-\alpha)f$ is operator $(k+1)$-tone on $(a,b)$.
\item[\rm(b)] For any $l\in \bN$, $f$ is operator $(k+2l)$-tone on $(a,b)$.
\item[\rm(c)] If $k$ is even, then ${d^k\over dt^k}\,f(A+tX)\big|_{t=0}\ge0$ for all
$A\in B(\cH)^{sa}(a,b)$ and $X\in B(\cH)^{sa}$.
\end{itemize}
\end{cor}

\begin{proof}
(a) has already been shown in the proof of Theorem \ref{T-3.3} (see the
paragraph after \eqref{eq-3.3}), and (b) is similarly shown by using Lemma \ref{L-3.2}.
Let us prove (c). For $A\in\bM_n^{sa}(a,b)$ and $X\in\bM_n^{sa}$ for any $n\in\bN$,
this is immediately seen from Corollary \ref{C-3.4}\,(iii) and \eqref{eq-3.1} in the proof
of Lemma \ref{L-3.2}. Next, let $A\in B(\cH)^{sa}(a,b)$ and $X\in B(\cH)^{sa}$.
Choose an orthogonal basis $\{e_n\}_{n=1}^\infty$ of $\cH$, and let $\cH_n$
be the linear span of $e_1,\dots,e_n$ and $P_n$ the orthogonal projection from $\cH$
onto $\cH_n$. Set $\tilde A_n:=P_nAP_n|_{\cH_n}$ and $\tilde X_n:=P_nXP_n|_{\cH_n}$,
which are considered as elements of $\bM_n^{sa}(a,b)$ and of $\bM_n^{sa}$, respectively.
By the matrix case shown above we have
${d^k\over dt^k}\,f(\tilde A_n+t\tilde X_n)\big|_{t=0}\ge0$.
Using \eqref{eq-3.2} one can easily verify that
$P_n\bigl({d^k\over dt^k}\,f(\tilde A_n+t\tilde X_n)\big|_{t=0}\bigr)P_n$ converges
strongly to ${d^k\over dt^k}\,f(A+tX)\big|_{t=0}$ as $n\to\infty$. Hence the
latter $k$th derivative is non-negative.
\end{proof}

\begin{remark}\rm
The phenomenon in the above (c) was observed in the paper \cite{Kr} for $k=2$. For the
proof one can alternatively use the higher derivative formula for infinite dimensional
operators in \cite{Pe} to verify the same expression as \eqref{eq-3.1}.
\end{remark}

It is well known that any operator monotone function on the whole line
$(-\infty,\infty)$ is a linear function $\alpha+\beta x$ with $\beta\ge0$; any operator
convex function on $(-\infty,\infty)$ is a quadratic function $\alpha+\beta x+\gamma x^2$
with $\gamma\ge0$. The following higher order extension is immediate from (iii) of
Corollary \ref{C-3.4}

\begin{cor}
Let $k\in\bN$. A real function on $(-\infty,\infty)$ is operator $k$-tone if and only if
it is a polynomial of degree less than or equal to $k$ with non-negative coefficient of
$x^k$.
\end{cor}

The following Corollary is also immediate from Theorem \ref{T-3.3}.

\begin{cor}
Let $f$ be a real function on $(a,b)$, where $-\infty\le a<b\le\infty$. Let
$k,m\in\bN$ with $m<k$. Then the following conditions are equivalent:
\begin{itemize}
\item[\rm(i)] $f$ is operator $k$-tone on $(a,b)$;
\item[\rm(ii)] $f$ is $C^m$ on $(a,b)$ and $f^{[m]}(x,\alpha_1,\dots,\alpha_m)$ is operator
$(k-m)$-tone on $(a,b)$ for every $\alpha_1,\dots,\alpha_m\in(a,b)$;
\item[\rm(iii)] $f$ is $C^{m-1}$ on $(a,b)$ and $f^{[m]}(x,\alpha,\dots,\alpha)$ is
operator $(k-m)$-tone on $(a,b)$ for some $\alpha\in(a,b)$ {\rm(}with continuation of value
at $x=\alpha${\rm)}.
\end{itemize}
\end{cor}

Let $\cF(a,b)$ denote the space of all real functions on $(a,b)$, which is a locally
convex topological vector space equipped with the pointwise convergence topology.
For each $k\in\bN$ we denote by $\cP^{(k)}(a,b)$ the set of all operator
$k$-tone functions on $(a,b)$.

\begin{prop}\label{P-3.9}
The set $\cP^{(k)}(a,b)$ is a closed convex cone in $\cF(a,b)$ for every $k\in\bN$ and
$$
\cP^{(1)}(a,b)\subsetneqq\cP^{(3)}(a,b)\subsetneqq\cP^{(5)}(a,b)\subsetneqq\cdots,
$$
$$
\cP^{(2)}(a,b)\subsetneqq\cP^{(4)}(a,b)\subsetneqq\cP^{(6)}(a,b)\subsetneqq\cdots.
$$
\end{prop}

\begin{proof}
Lemma \ref{L-1.1} clearly implies that the set $\cP^{(k)}_n(a,b)$ of all matrix $k$-tone
functions of order $n$ on $(a,b)$ is a positive cone and is closed for the pointwise
convergence. Hence $\cP^{(k)}(a,b)=\bigcap_{n\geq 0}\cP^{(k)}_n(a,b)$ has the same
properties.

The inclusion have already been seen in Corollary \ref{C-3.5}\,(b). Moreover,
for each $k\in\bN$ with $k\ge2$, let $f(x)=(x-\alpha)^k$ with $\alpha\in(a,b)$. Then
$$
f^{[k-1]}(x_1,\dots,x_k)=\sum_{i=1}^k(x_i-\alpha),\qquad
f^{[k+1]}(x_1,\dots,x_{k+2})\equiv0.
$$
Hence $f\in\cP^{(k+1)}(a,b)$ but $f\not\in\cP^{(k-1)}(a,b)$ since
$f^{[k-1]}(x_1,\dots,x_k)$ is negative for $x_1,\dots,x_k\in(a,\alpha)$. (More intrinsic
differences among $\cP^{(k)}(a,b)$'s will be seen from Proposition \ref{P-4.3} and
examples in Section \ref{sec-6}.)
\end{proof}

Recall \cite[p.\ 131]{Do} that $n$-monotone functions have the following local property:
Let $a<c<b<d$ and $f$ be a real function on $(a,d)$. If $f|_{(a,b)}$ and $f|_{(c,d)}$ are
$n$-monotone, then so is $f$ on $(a,d)$. This property is essential in the proof
\cite[pp.\ 83--84]{Do} of the fact that (d) of Theorem \ref{T-2.1} with $k=1$ is a
sufficient condition for $n$-monotone functions. This property for $n$-convex
functions is still open while that for $2$-convex functions was proved in \cite{HT1}.
It is immediate from (vii) or (v) of Theorem \ref{T-3.3} that operator $k$-tone functions
have a similar local property for every $k\in\bN$.

\section{Operator $k$-tone functions on $(-1,1)$}\label{sec-4}

In this section we discuss operator $k$-tone functions restricted on the domain interval
$(-1,1)$. In the next theorem we show further characterizations of such functions by
using Theorem \ref{T-3.3} and the integral representation of operator monotone functions
in Theorem \ref{T-1.8} (in fact, \eqref{eq-1.5} is \eqref{eq-4.1} when $k=1$).

\begin{thm}\label{T-4.1}
Let $f$ be a real function on $(-1,1)$, and let $k\in\bN$. Then the following
conditions {\rm(i)}--{\rm(iii)} are equivalent:
\begin{itemize}
\item[\rm(i)] $f$ is operator $k$-tone on $(-1,1)$;
\item[\rm(ii)] $f$ is $C^{k-1}$ on $(-1,1)$ and there exists a finite positive measure
$\mu$ on $[-1,1]$ such that for any choice of $\alpha\in(-1,1)$,
\begin{equation}\label{eq-4.1}
f(x)=\sum_{l=0}^{k-1}{f^{(l)}(\alpha)\over l!}(x-\alpha)^l
+\int_{[-1,1]}{(x-\alpha)^k\over(1-\lambda x)(1-\lambda\alpha)^k}\,d\mu(\lambda),
\quad x\in(-1,1);
\end{equation}
\item[\rm(iii)] $f$ is $C^k$ on $(-1,1)$ and there exists a finite
positive measure $\mu$ on $[-1,1]$ such that
\begin{equation}\label{eq-4.2}
f^{[k]}(x_1,x_2,\dots,x_{k+1})=\int_{[-1,1]}
{1\over(1-\lambda x_1)(1-\lambda x_2)\cdots(1-\lambda x_{k+1})}\,d\mu(\lambda)
\end{equation}
for every $x_1,x_2,\dots,x_{k+1}\in(-1,1)$.
\end{itemize}

Moreover, in the above situation, the measures $\mu$ in {\rm(ii)} and in {\rm(iii)} are
unique and same, and the following hold for every $m>k$:
\begin{itemize}
\item[\rm(a)] For any choice of $\alpha\in(-1,1)$,
\begin{equation}\label{eq-4.3}
f(x)=\sum_{l=0}^{m-1}{f^{(l)}(\alpha)\over l!}(x-\alpha)^l
+\int_{[-1,1]}{(x-\alpha)^m\lambda^{m-k}\over(1-\lambda x)(1-\lambda\alpha)^m}
\,d\mu(\lambda),\quad x\in(-1,1).
\end{equation}
\item[\rm(b)] For every $x_1,x_2,\dots,x_{m+1}\in(-1,1)$,
$$
f^{[m]}(x_1,x_2,\dots,x_{m+1})=\int_{[-1,1]}
{\lambda^{m-k}\over(1-\lambda x_1)(1-\lambda x_2)\cdots(1-\lambda x_{m+1})}
\,d\mu(\lambda).
$$
\end{itemize}
\end{thm}

\begin{proof}
(i)\,$\Rightarrow$\,(ii).\enspace For each $\alpha\in(-1,1)$ use Corollary
\ref{C-3.4}\,(iii) to have an operator monotone function $g$ on $(-1,1)$, which is
represented by Theorem \ref{T-1.8} as in \eqref{eq-1.5} with a representing measure $\mu$.
Replacing $(1-\lambda\alpha)^{k-1}d\mu(\lambda)$ with $d\mu(\lambda)$ we have expression
\eqref{eq-4.1}. We next prove that the measure $\mu$ does not depend on $\alpha$.
Let $\tilde\alpha\in(-1,1)$ be arbitrary. Since
$$
{(x-\alpha)^k\over(1-\lambda x)(1-\lambda\alpha)^k}
-{(x-\tilde\alpha)^k\over(1-\lambda x)(1-\lambda\tilde\alpha)^k}
={(x-\alpha)^k(1-\lambda\tilde\alpha)^k-(x-\tilde\alpha)^k(1-\lambda\alpha)^k
\over(1-\lambda x)(1-\lambda\alpha)^k(1-\lambda\tilde\alpha)^k}
$$
and the numerator of the above right-hand side is zero when $x=1/\lambda$, we notice that
the above expression is written in the form
$$
P_{k-1}(x):=\sum_{l=0}^{k-1}a_l(\lambda)x^l,
$$
where $a_l(\lambda)$, $0\le l\le k-1$, are functions of $\lambda$ ($\alpha,\tilde\alpha$
being constants). Since $P_{k-1}(0)$, $P_{k-1}'(0)$, $\dots$, $P_{k-1}^{(k-1)}(0)$ are
functions of $\lambda$ on $[-1,1]$ integrable with respect to $\mu$, we notice that
$a_l(\lambda)$, $0\le l\le k-1$, are integrable with respect to $\mu$. Therefore, we have
$$
f(x)=(\mbox{a polynomial of at most degree $k-1$})
+\int_{[-1,1]}{(x-\tilde\alpha)^k\over(1-\lambda x)(1-\lambda\tilde\alpha)^k}
\,d\mu(\lambda).
$$
Since the above integral term is $o\bigl((x-\tilde\alpha)^{k-1}\bigr)$, the first
polynomial term must be given by the Taylor formula. Thus \eqref{eq-4.1} is valid with
$\tilde\alpha$ in place of $\alpha$.

The proof of (ii)\,$\Rightarrow$\,(iii), as well as that of (b) from \eqref{eq-4.1}, is
included in the proof of Corollary \ref{C-1.9}, and (iii)\,$\Rightarrow$\,(i) was actually
shown in the proof of Lemma \ref{L-3.2}. Moreover, (a) follows from (b) by letting
$x_1=x$ and $x_2=\dots=x_{m+1}=\alpha$.

It remains to prove the uniqueness of $\mu$ in (iii). Recall that the linear span of
functions $h_x(\lambda):=1/(1-\lambda x)$ where $x\in (-1,1)$ is dense in $C([-1,1])$,
the space of continuous functions on $[-1,1]$. So, letting $x_1=x$ and
$x_2=\dots=x_{n+2}=0$ in \eqref{eq-4.2} one can easily see that $\mu$ is unique.
\end{proof}

Of course the integral term of \eqref{eq-4.3} is the $m$th remainder term of the
Taylor series of $f$ at $\alpha$. This remainder term converges to $0$ as
$m\to\infty$ for $|x-\alpha|<1-|\alpha|$ by Lemma \ref{L-3.1} 
(this follows also by the Lebesgue convergence theorem).

We call the finite measure $\mu$ on $[-1,1]$ in (ii) and (iii) of Theorem \ref{T-4.1} the
{\it representing measure} of $f$. Theorem \ref{T-4.1}\,(a) says that if
$f\in\cP^{(k)}(-1,1)$ with the representing measure $\mu$ and if $m>k$ and $m-k$ is even,
then $f\in\cP^{(m)}(-1,1)$ with the representing measure $\lambda^{m-k}\,d\mu(\lambda)$.
In this connection we show

\begin{prop}\label{P-4.2}
Let $f\in\cP^{(k)}(-1,1)$ with the representing measure $\mu$. Then
\begin{itemize}
\item[\rm(1)] $f\in\cP^{(k+1)}(-1,1)$ if and only if $\mu$ is supported in $[0,1]$. In this
case, $f\in\cP^{(m)}(-1,1)$ for all $m>k$.
\item[\rm(2)] $-f\in\cP^{(k+1)}(-1,1)$ if and only if $\mu$ is supported in $[-1,0]$.
In this case, $(-1)^{m-k}f\in\cP^{(m)}(-1,1)$ for all $m>k$.
\end{itemize}
\end{prop}

\begin{proof}
(1)\enspace
Assume that $f\in\cP^{(k+1)}(-1,1)$. Theorem \ref{T-4.1} implies that there exists a
finite positive measure $\mu'$ on $[-1,1]$ such that
$$
f^{[k+1]}(x_1,\dots,x_{k+2})
=\int_{[-1,1]}{1\over(1-\lambda x_1)\cdots(1-\lambda x_{k+2})}\,d\mu'(\lambda)
$$
for every $x_1,\dots,x_{k+2}\in(-1,1)$. On the other hand, by Theorem \ref{T-4.1}\,(b) we
have
$$
f^{[k+1]}(x_1,\dots,x_{k+2})
=\int_{[-1,1]}{\lambda\over(1-\lambda x_1)\cdots(1-\lambda x_{k+2})}\,d\mu(\lambda)
$$
for every $x_1,\dots,x_{k+2}\in(-1,1)$. As in the last part of the proof of
Theorem \ref{T-4.1}, letting $x_1=x$ and $x_2=\dots=x_{n+2}=0$, we must have
$d\mu'(\lambda)=\lambda\,d\mu(\lambda)$ on $[-1,1]$. This means that $\mu$ is supported
in $[0,1]$. Conversely, if $\mu$ is supported in $[0,1]$, then $f\in\cP^{(m)}(-1,1)$ for
all $m\ge k$ thanks to Theorem \ref{T-4.1}.

(2) is similarly shown.
\end{proof}

In case of $m=1$ (or $m=2$) the next proposition shows when $f\in\cP^{(k)}(-1,1)$ is a sum
of a polynomial of degree $k$ and an operator monotone (or operator convex) function on
$(-1,1)$.

\begin{prop}\label{P-4.3}
Let $k,m\in\bN$ with $m<k$. Let $f\in\cP^{(k)}(-1,1)$ with the representing measure $\mu$.
\begin{itemize}
\item[\rm(1)] Assume that $k-m$ is even. Then there exist $a_0,\dots,a_k\in\bR$ with
$a_k\ge0$ and $g\in\cP^{(m)}(-1,1)$ such that
$$
f(x)=\sum_{l=0}^ka_lx^l+g(x),\qquad x\in(-1,1),
$$
if and only if
$$
\int_{[-1,1]\setminus\{0\}}{1\over\lambda^{k-m}}\,d\mu(\lambda)<+\infty.
$$
\item[\rm(2)] Assume that $k-m$ is odd and let $\varepsilon=\pm 1$. Then there exist $a_0,\dots,a_k\in\bR$ with
$a_k\ge0$ and $g\in\cP^{(m)}(-1,1)$ such that
$$
f(x)=\sum_{l=0}^ka_lx^l+\, \varepsilon \,g(x),\qquad x\in(-1,1),
$$
if and only if $\mu$ is supported in $\varepsilon [0,1]$ and
$$
\int_{\varepsilon(0,1]} \frac 1{|\lambda|^{k-m}}\,d\mu(\lambda)<+\infty.
$$
\end{itemize}
\end{prop}

\begin{proof}
(1)\enspace
Assume that $f$ is of the form in (1). Let $\mu'$ be the representing measure of $g$. By
Theorem \ref{T-4.1}\,(a) we then have
\begin{align*}
f^{[k]}(x_1,x_2,\dots,x_{k+1})
&=a_k+\int_{[-1,1]}
{\lambda^{k-m}\over(1-\lambda x_1)(1-\lambda x_2)\cdots(1-\lambda x_{k+1})}
\,d\mu'(\lambda) \\
&=\int_{[-1,1]}
{1\over(1-\lambda x_1)(1-\lambda x_2)\cdots(1-\lambda x_{k+1})}\,d\mu(\lambda).
\end{align*}
As in the proof of Proposition \ref{P-4.2} we have
\begin{equation}\label{eq-4.4}
d\mu(\lambda)=a_k\,d\delta_0(\lambda)+\lambda^{k-m}\,d\mu'(\lambda)
\end{equation}
so that
$$
\int_{[-1,1]\setminus\{0\}}{1\over\lambda^{k-m}}\,d\mu(\lambda)
=\int_{[-1,1]\setminus\{0\}}d\mu'(\lambda)<+\infty.
$$

Conversely, assume that
$\int_{[-1,1]\setminus\{0\}}1/\lambda^{k-m}\,d\mu(\lambda)<+\infty$. Since
\begin{align*}
{x^k\over1-\lambda x}
&={x^m\{1-(1-\lambda x)\}^{k-m}\over(1-\lambda x)\lambda^{k-m}} \\
&=x^m\sum_{l=1}^{k-m}{k-m\choose l}(-1)^l\,{(1-\lambda x)^{l-1}\over
\lambda^{k-m}}+{x^m\over(1-\lambda x)\lambda^{k-m}},\qquad\lambda\ne0,
\end{align*}
we have
\begin{align*}
f(x)&=\sum_{l=0}^{k-1}{f^{(l)}(0)\over l!}x^l
+\int_{[-1,1]}{x^k\over1-\lambda x}\,d\mu(\lambda) \\
&=(\mbox{a polynomial of at most degree $k$})
+\int_{[-1,1]\setminus\{0\}}{x^m\over(1-\lambda x)\lambda^{k-m}}\,d\mu(\lambda)
\end{align*}
where
$$
g(x):=\int_{[-1,1]\setminus\{0\}}{x^m\over(1-\lambda x)\lambda^{k-m}}\,d\mu(\lambda)
$$
belongs to $\cP^{(m)}(-1,1)$. In the above, $g(x)$ is well defined by the integrability
assumption.

(2)\enspace We do it only for $\varepsilon=1$.
Assume that $f$ is of the form in (2). Let $\mu'$ be the representing measure of $g$.
Then we have \eqref{eq-4.4} as in the proof of (1). Since $k-m$ is odd, $\mu$ and $\mu'$
are supported in $[0,1]$ and
$$
\int_{(0,1]}{1\over\lambda^{k-m}}\,d\mu(\lambda)=\int_{(0,1]}d\mu'(\lambda)<+\infty.
$$
The proof of the converse implication is also similar to that of (1) by replacing the
integral region $[-1,1]\setminus\{0\}$ with $(0,1]$.
\end{proof}

We say that a smooth real function $f$ on $(a,b)$ is {\it operator absolutely monotone} if
\begin{equation}\label{eq-4.5}
{d^k\over dt^k}\,f(A+tX)\Big|_{t=0}\ge0,\qquad k=0,1,2,\dots,
\end{equation}
for every $A\in\bM_n^{sa}(a,b)$ and every $X\in\bM_n^+$ for any $n\in\bN$, and
{\it operator completely monotone} if
\begin{equation}\label{eq-4.6}
(-1)^k{d^k\over dt^k}\,f(A+tX)\Big|_{t=0}\ge0,\qquad k=0,1,2,\dots,
\end{equation}
for every $A\in\bM_n^{sa}(a,b)$ and every $X\in\bM_n^+$ for any $n\in\bN$. From
Proposition \ref{P-4.2} one can characterize these operator versions of
absolutely/completely monotone functions on $(-1,1)$ as follows:

\begin{prop}\label{P-4.4}
The following conditions for a real function on $(-1,1)$ are equivalent:
\begin{itemize}
\item[\rm(i)] $f$ is operator absolutely monotone;
\item[\rm(ii)] $f$ is a non-negative operator monotone function whose representing measure
is supported in $[0,1]$;
\item[\rm(iii)] $f$ admits the integral expression
$$
f(x)=\beta+\int_{[0,1]}{1+x\over1-\lambda x}\,d\mu(\lambda)
$$
with $\beta\ge0$ and a {\rm(}unique{\rm)} finite positive measure $\mu$ on $[0,1]$. 
\end{itemize}
\end{prop}

\begin{proof}
(i) $\Leftrightarrow$ (ii) is obvious from Proposition \ref{P-4.2}\,(1) for $k=1$. 
(ii)\,$\Rightarrow$\,(iii) follows by letting $\alpha\searrow-1$ in \eqref{eq-1.5} of Theorem
\ref{T-1.8} and then replacing $(1+\lambda)^{-1}\,d\mu(\lambda)$ with $d\mu(\lambda)$.
(iii)\,$\Rightarrow$\,(ii) is immediate since $(1+x)/(1-\lambda x)$ is non-negative and
operator monotone on $(-1,1)$ for $\lambda\in[0,1]$.
\end{proof}

\begin{prop}
The following conditions for a real function on $(-1,1)$ are equivalent:
\begin{itemize}
\item[\rm(i)] $f$ is operator completely monotone;
\item[\rm(ii)] $-f$ is a non-positive operator monotone function whose representing measure
is supported in $[-1,0]$;
\item[\rm(iii)] $f$ admits the integral expression
$$
f(x)=\beta+\int_{[-1,0]}{1-x\over1-\lambda x}\,d\mu(\lambda)
$$
with $\beta\ge0$ and a {\rm(}unique{\rm)} finite positive measure $\mu$ on $[-1,0]$.
\end{itemize}
\end{prop}

\begin{proof}
The proof is similar to that of Proposition \ref{P-4.4}. Alternatively, this proposition
immediately follows from Proposition \ref{P-4.4} since $f$ is operator completely monotone
on $(-1,1)$ if and only if $f(-x)$ is operator absolutely monotone on $(-1,1)$.
\end{proof}

Furthermore, one can see from Theorem \ref{T-3.3} that if $f$ is operator absolutely
monotone (resp., operator completely monotone) on $(-1,1)$, then \eqref{eq-4.5} (resp.,
\eqref{eq-4.6}) holds for every $A\in B(\cH)^{sa}(-1,1)$ and $X\in B(\cH)^+$. This
justifies our terminology.

\section{Operator $k$-tone functions on $(0,\infty)$}\label{sec-5}

In addition to general characterizations in Theorem \ref{T-3.3}, the next theorem
gives integral representation of operator $k$-tone functions on the unbounded interval
$(0,\infty)$. Note that \eqref{eq-1.8} is \eqref{eq-5.2} when $k=1$.

\begin{thm}\label{T-5.1}
Let $f$ be a real function on $(0,\infty)$, and let $k\in\bN$. Then the following
conditions {\rm(i)}--{\rm(iii)} are equivalent:
\begin{itemize}
\item[\rm(i)] $f$ is operator $k$-tone on $(0,\infty)$;
\item[\rm(ii)] $f$ is $C^{k-1}$ on $(0,\infty)$ and there exist a $\gamma\ge0$ and a
positive measure $\mu$ on $[0,\infty)$ such that
\begin{equation}\label{eq-5.1}
\int_{[0,\infty)}{1\over(1+\lambda)^{k+1}}\,d\mu(\lambda)<+\infty
\end{equation}
and for any choice of $\alpha\in(0,\infty)$,
\begin{equation}\label{eq-5.2}
f(x)=\sum_{l=0}^{k-1}{f^{(l)}(\alpha)\over l!}\,(x-\alpha)^l+\gamma(x-\alpha)^k
\int_{[0,\infty)}{(x-\alpha)^k\over(x+\lambda)(\alpha+\lambda)^k}\,d\mu(\lambda),
\quad x\in(0,\infty);
\end{equation}
\item[\rm(iii)] $f$ is $C^k$ on $(0,\infty)$ and there exist a $\gamma\ge0$ and a positive
measure $\mu$ on $[0,\infty)$ such that \eqref{eq-5.1} holds and
\begin{equation}\label{eq-5.3}
f^{[k]}(x_1,x_2,\dots,x_{k+1})=\gamma+\int_{[0,\infty)}
{1\over(x_1+\lambda)(x_2+\lambda)\cdots(x_{k+1}+\lambda)}\,d\mu(\lambda)
\end{equation}
for every $x_1,x_2,\dots,x_{k+1}\in(0,\infty)$.
\end{itemize}

Moreover, in the above situation, $\gamma$ and $\mu$ in {\rm(ii)} and in {\rm(iii)} are
unique and same, and the following hold for every $m>k$:
\begin{itemize}
\item[\rm(a)] For any choice of $\alpha\in(0,\infty)$,
\begin{equation}\label{eq-5.4}
f(x)=\sum_{l=0}^{m-1}{f^{(l)}(\alpha)\over l!}(x-\alpha)^l
+(-1)^{m-k}\int_{[0,\infty)}{(x-\alpha)^m\over(x+\lambda)(\alpha+\lambda)^m}
\,d\mu(\lambda),\quad x\in(0,\infty),
\end{equation}
and hence $(-1)^{m-k}f$ is operator $m$-tone on $(0,\infty)$.
\item[\rm(b)] For every $x_1,x_2,\dots,x_{m+1}\in(0,\infty)$,
$$
f^{[m]}(x_1,x_2,\dots,x_{m+1})=(-1)^{m-k}\int_{[0,\infty)}
{1\over(x_1+\lambda)(x_2+\lambda)\cdots(x_{m+1}+\lambda)}\,d\mu(\lambda).
$$
\end{itemize}
\end{thm}

\begin{proof}
(i)\,$\Rightarrow$\,(ii).\enspace
Assume (i); by Theorem \ref{T-3.3}, $f$ is analytic and
$g(x):=f^{[k-1]}(x,\alpha,\dots,\alpha)$ is operator monotone on $(0,\infty)$ for any
fixed $\alpha\in(0,\infty)$. Apply Theorem \ref{T-1.10} to $g$ and replace
$(\alpha+\lambda)^{k-1}\,d\mu(\lambda)$ with $d\mu(\lambda)$. Then, thanks to
\eqref{eq-3.3} with $(a,b)=(0,\infty)$, there exist a $\beta\in\bR$, a $\gamma\ge0$ and
a positive measure $\mu$ on $[0,\infty)$ satisfying \eqref{eq-5.1} such that
\begin{align*}
f(x)&=\sum_{l=0}^{k-2}{f^{(l)}(\alpha)\over l!}\,(x-\alpha)^l
+\beta(x-\alpha)^{k-1}+\gamma(x-\alpha)^k \\
&\qquad\qquad\qquad
+\int_{[0,\infty)}{(x-\alpha)^k\over(x+\lambda)(\alpha+\lambda)^k}\,d\mu(\lambda),
\qquad x\in(0,\infty).
\end{align*}
Since the above integral term is $o\bigl((x-\alpha)^{k-1}\bigr)$, we have
$\beta=f^{(k-1)}(\alpha)/(k-1)!$ so that expression \eqref{eq-5.2} holds for $\alpha$ fixed
above. Let $\tilde\alpha\in(0,\infty)$ be arbitrary. Since
$$
{(x-\alpha)^k\over(x+\lambda)(\alpha+\lambda)^k}
-{(x-\tilde\alpha)^k\over(x+\lambda)(\tilde\alpha+\lambda)^k}
$$
is a polynomial in $x$ (with coefficients depending on $\lambda$) of at most
degree $k-1$, one can show as in the proof of (i)\,$\Rightarrow$\,(ii) of
Theorem \ref{T-4.1} that the measure $\mu$ does not depend on $\alpha$.

(ii)\,$\Rightarrow$\,(iii).\enspace
For every $\lambda\in[0,\infty)$, similarly to \eqref{eq-1.a} we have
$$
{(x-\alpha)^k\over(x+\lambda)(\alpha+\lambda)^k}
=(\mbox{a polynomial of degree $k-1$})+{(-1)^k\over x+\lambda}
$$
so that
\begin{align*}
&\biggl({(x-\alpha)^k\over(x+\lambda)(\alpha+\lambda)^k}\biggr)^{[k]}
(x_1,x_2,\dots,x_{k+1}) 
={1\over(x_1+\lambda)(x_2+\lambda)\cdots(x_{k+1}+\lambda)}
\end{align*}
for every $x_1,x_2,\dots,x_{k+1}\in(0,\infty)$. Hence (iii) follows from (ii) by taking the
$k$th divided differences of both sides of \eqref{eq-5.2}.

(iii)\,$\Rightarrow$\,(i).\enspace
Let $A=\diag(a_1,\dots,a_n)\in\bM_n^{sa}(0,\infty)$ and $X\in\bM_n^+$. Similarly to the
proof of Lemma \ref{L-3.2}, based on Daleckii and Krein's derivative formula, one can
show from (iii) that
$$
{d^k\over dt^k}\,f(A+tX)\Big|_{t=0}
=k!\int_{[0,\infty)}D(\lambda)^{1/2}(D(\lambda)^{1/2}XD(\lambda)^{1/2})^k
D(\lambda)^{1/2}\,d\mu(\lambda)\ge0,
$$
with
$$
D(\lambda):=\diag\biggl({1\over a_1+\lambda},\dots,{1\over a_n+\lambda}\biggr),
\qquad\lambda\in[0,\infty).
$$
This yields (i) by Theorem \ref{T-3.3}.

Next, we prove the uniqueness of $\gamma$ and $\mu$ in (ii) or in (iii). It suffices to
show the uniqueness of $\gamma$ and $\mu$ in (iii). Let $x_1=x$ and $x_2=\dots=x_{k+1}=1$
in \eqref{eq-5.3}. Then
\begin{equation}\label{eq-5.5}
g(x):=f^{[k]}(x,1,\dots,1)
=\gamma+\int_{[0,\infty)}{1\over x+\lambda}\,d\nu(\lambda),
\end{equation}
where
$$
d\nu(\lambda):={1\over(1+\lambda)^k}\,d\mu(\lambda),\quad
\mbox{hence}\ \ \int_{[0,\infty)}{1\over1+\lambda}\,d\nu(\lambda)<+\infty.
$$
This says that $g$ is a non-negative operator monotone decreasing function on $(0,\infty)$.
It is well known that a $\gamma\ge0$ and a measure $\nu$ on $[0,\infty)$ representing $g$
in \eqref{eq-5.5} are unique (in fact, $\gamma=\lim_{x\to\infty}g(x)$).
So $\gamma$ and $\mu$ in (iii) are unique.

Finally, we prove (a) and (b). Assertion (a) immediately follows by computing higher order
divided differences from \eqref{eq-5.3}. Moreover, \eqref{eq-5.4} follows
from (b) by letting $x_1=\alpha$ and $x_2=\dots=x_{m+1}=\alpha$.
The operator $m$-tonicity of $(-1)^{m-k}f$ now follows from expression \eqref{eq-5.4}
due to condition (ii).
\end{proof}

One can understand the integral term of \eqref{eq-5.4} as the $m$th remainder term of the
Taylor series of $f$ at $\alpha$, which converges to $0$ as $m\to\infty$ for
$x\in(0,2\alpha)$.

\begin{remark}\rm
By using condition (iii) of Theorem \ref{T-5.1} in the cases $k=1,2$ and
\cite[Theorem 3.1]{AH} (also \cite{Ha}) we see that a real function $f$ on $(0,\infty)$ is
operator monotone (resp., operator convex) if and only if $f^{[1]}(x,\alpha)$ (resp.,
$f^{[2]}(x,\alpha,\alpha)$ with an additional condition of $f$ being $C^1$) is
non-negative, non-increasing (as a numerical function) and operator convex on $(0,\infty)$
for some $\alpha\in(0,\infty)$ (with continuation at $x=\alpha$). It is also known
\cite[Corollary 2.7.8]{Hi}, \cite[Lemma 2.1]{Uc} that a real function $f$ on $(a,b)$
is operator convex if and only if $f^{[1]}(x,\alpha)$ is operator monotone on $(a,b)$ for
some $\alpha\in(a,b)$ (with continuation at $x=\alpha$).
\end{remark}

Assertion (a) of Theorem \ref{T-5.1} gives the following inclusion property of
$\cP^{(k)}(0,\infty)$'s, where the strict inclusions are seen as in
the proof of Proposition \ref{P-3.9} (see also Proposition \ref{P-5.6} and examples in
Section \ref{sec-6}). The first inclusion $\cP^{(1)}(0,\infty)\subset-\cP^{(2)}(0,\infty)$
is essentially the same as the well-known fact \cite[V.2.5]{Bh} that operator monotone
functions on $[0,\infty)$ are operator concave there.

\begin{prop}\label{P-5.2}
The closed convex cones $\cP^{(k)}(0,\infty)$, $k\in\bN$, of $\cF(0,\infty)$ satisfy
$$
\cP^{(1)}(0,\infty)\subsetneqq-\cP^{(2)}(0,\infty)\subsetneqq\cP^{(3)}(0,\infty)
\subsetneqq-\cP^{(4)}(0,\infty)\subsetneqq\cdots,
$$
where $-\cP^{(k)}(0,\infty):=\{-f:f\in\cP^{(k)}(0,\infty)\}$.
\end{prop}

Theorem \ref{T-5.1} says that a function $f\in\cP^{(k)}(0,\infty)$ admits a unique integral
expression given in \eqref{eq-5.2} with a constant $\gamma\ge0$ and a measure $\mu$ on
$[0,\infty)$ satisfying \eqref{eq-5.1}. We call $\gamma$ the {\it coefficient of the $k$th
degree} of $f$ and $\mu$ the {\it representing measure} of $f$.

In case of $m=1$ (or $m=2$) the next proposition characterizes when
$f\in\cP^{(k)}(0,\infty)$ is a sum or difference of a polynomial of degree $k$ and an
operator monotone (or operator convex) function on $(0,\infty)$.

\begin{prop}\label{P-5.6}
Let $k,m\in\bN$ with $m<k$. For a real function $f$ on $(0,\infty)$ the following are
equivalent:
\begin{itemize}
\item[\rm(i)] there exist $a_0,\dots,a_k\in\bR$ with $a_k\ge0$ and
$g\in\cP^{(m)}(0,\infty)$ such that
$$
f(x)=\sum_{l=0}^ka_lx^l+(-1)^{k-m}g(x),\qquad x\in(0,\infty);
$$
\item[\rm(ii)] $f\in\cP^{(k)}(0,\infty)$ and the representing measure $\mu$ of $f$
satisfies
\begin{equation}\label{eq-5.6}
\int_{[0,\infty)}{1\over(1+\lambda)^{m+1}}\,d\mu(\lambda)<+\infty.
\end{equation}
\end{itemize}
\end{prop}

\begin{proof}
(i)\,$\Rightarrow$\,(ii).\enspace
Assume that $f$ is of the form in (i). Let $\mu$ be the representing measure of $g$, which
satisfies \eqref{eq-5.6}. Then by (b) of Theorem \ref{T-5.1} we have
$$
f^{[k]}(x_1,x_2,\dots,x_{k+1})=a_k+\int_{[0,\infty)}
{1\over(x_1+\lambda)(x_2+\lambda)\cdots(x_{k+1}+\lambda)}\,d\mu(\lambda).
$$
Hence Theorem \ref{T-5.1} implies that $f\in\cP^{(k)}(0,\infty)$ with the coefficient of
the $k$th degree $a_k$ and the representing measure $\mu$.

(ii)\,$\Rightarrow$\,(i).\enspace
Assume that $f$ is represented as \eqref{eq-5.2} with the  measure $\mu$ satisfying
\eqref{eq-5.6}. Similarly to the proof of Proposition \ref{P-4.3}\,(1) we have
\begin{align*}
f(x)&=(\mbox{a polynomial of at most degree $k$})+(-1)^{k-m}g(x)
\end{align*}
with
$$
g(x):=\int_{[0,\infty)}{(x-\alpha)^m\over
(x+\lambda)(\alpha+\lambda)^m}\,d\mu(\lambda)
$$
belonging to $\cP^{(m)}(0,\infty)$.
\end{proof}

The case $m=0$ version of Proposition \ref{P-5.6} can be stated as follows: A function
$f\in\cP^{(k)}(0,\infty)$ with the representing measure $\mu$ is of the form
$f(x)=a_0+(-1)^kg(x)$ with $a_0\ge0$ and a non-negative operator monotone decreasing
function $g$ on $(0,\infty)$ if and only if
$$
\int_{[0,\infty)}{1\over1+\lambda}\,d\mu(\lambda)<+\infty.
$$
The proof is similar to the above. This suggests us to define $\cP^{(0)}(0,\infty)$ as the
set of all non-negative operator monotone decreasing functions on $(0,\infty)$; then
$\cP^{(0)}(0,\infty)\subset-\cP^{(1)}(0,\infty)\subset\cP^{(2)}(0,\infty)$.

\begin{lemma}\label{L-5.6}
Let $k\in\bN$ and $f\in\cP^{(k)}(0,\infty)$. Then $\lim_{x\searrow0}xf(x)$
and $\lim_{x\to\infty}f(x)/x^k$ exist and
$$
\lim_{x\searrow0}xf(x)\in\begin{cases}
(-\infty,0] & \text{if $k$ is odd}, \\
[0,\infty) & \text{if $k$ is even},
\end{cases}\qquad
\lim_{x\to\infty}{f(x)\over x^k}\in[0,\infty).
$$
\end{lemma}

\begin{proof}
First, assume that $g$ is an operator monotone function on $(0,\infty)$. From the integral
representation of $g$, it is easy to show that $\lim_{x\searrow0}xg(x)$ and
$\lim_{x\to\infty}g(x)/x$ exist and
\begin{equation}\label{eq-5.7}
\lim_{x\searrow0}xg(x)\in(-\infty,0],\qquad\lim_{x\to\infty}{g(x)\over x}\in[0,\infty),
\end{equation}
see also \cite[Corollary 2.7]{HS}. Next, assume that $f\in\cP^{(k)}(0,\infty)$. By
Corollary \ref{C-3.4} there is an operator monotone function $g$ on $(0,\infty)$ such
that $f$ is a sum of a polynomial of degree less than or equal to $k-2$ and
$(x-1)^{k-1}g(x)$. This together with \eqref{eq-5.7} yields the conclusion. (In fact, by
using the Lebesgue convergence theorem, one can easily see from \eqref{eq-5.2} that
$\gamma=\lim_{x\to\infty}f(x)/x^k$, the coefficient of $k$th degree of $f$.)
\end{proof}

\begin{prop}
Let $k,m\in\bN$ with $k<m$. For a real function $f$ on $(0,\infty)$ the following are
equivalent:
\begin{itemize}
\item[\rm(i)] $f\in\cP^{(k)}(0,\infty)$;
\item[\rm(ii)] $(-1)^{m-k}f\in\cP^{(m)}(0,\infty)$ and
$\lim_{x\to\infty}f(x)/x^k\in[0,\infty)$.
\end{itemize}
\end{prop}

\begin{proof}
(i)\,$\Rightarrow$\,(ii).\enspace
Assume that $f\in\cP^{(k)}(0,\infty)$, and let $\gamma\ge0$ be the coefficient of the
$k$th degree and $\mu$ the representing measure of $f$. Then
$(-1)^{m-k}f\in\cP^{(m)}(0,\infty)$ by Proposition \ref{P-5.2}, and Lemma \ref{L-5.6}
implies that $\lim_{x\to\infty}f(x)/x^k\in[0,\infty)$.

(ii)\,$\Rightarrow$\,(i).\enspace
It suffices to prove the case where $k=m-1$, that is,
\begin{itemize}
\item[($*$)] if $-f\in\cP^{(m)}(0,\infty)$ and $\lim_{x\to\infty}f(x)/x^{m-1}\in[0,\infty)$, then
$f\in\cP^{(m-1)}(0,\infty)$.
\end{itemize}
Indeed, assume that (ii) holds for $k<m$. Since
$\lim_{x\to\infty}(-1)^{m-k-1}f(x)/x^{m-1}\in[0,\infty)$ (in fact,
$\lim_{x\to\infty}f(x)/x^{m-1}=0$ if $k<m-1$), we apply $(*)$ to have
$(-1)^{m-k-1}f\in\cP^{(m-1)}(0,\infty)$. If $k<m-1$, then we apply $(*)$ again to have
$(-1)^{m-k-2}f\in\cP^{(m-2)}(0,\infty)$. Repeating this procedure yields that
$f\in\cP^{(k)}(0,\infty)$.

To prove $(*)$, assume that $-f\in\cP^{(m)}(0,\infty)$ and
$\lim_{x\to\infty}f(x)/x^{m-1}\in[0,\infty)$. Let $\gamma_0\ge0$ and $\mu$ be the
coefficient of the $m$th degree and the representing measure of $-f$, respectively. With
$\alpha\in(0,\infty)$ we have
$$
f(x)=\sum_{l=0}^{m-1}{f^{(l)}(\alpha)\over l!}\,(x-\alpha)^l-\gamma_0(x-\alpha)^m
-\int_{[0,\infty)}{(x-\alpha)^m\over(x+\lambda)(\alpha+\lambda)^m}\,d\mu(\lambda),
\quad x\in(0,\infty).
$$
Since
\begin{align*}
\lim_{x\to\infty}{1\over x^{m-1}}\int_{[0,\infty)}
{(x-\alpha)^m\over(x+\lambda)(\alpha+\lambda)^m}\,d\mu(\lambda)
&=\int_{[0,\infty)}{1\over(\alpha+\lambda)^m}\,d\mu(\lambda)
\end{align*}
by the monotone convergence theorem, the assumption
$\lim_{x\to\infty}f(x)/x^{m-1}\in[0,\infty)$ implies that $\gamma_0=0$ and
\begin{equation}\label{eq-5.8}
\gamma_1:={f^{(m-1)}(\alpha)\over(m-1)!}
-\int_{[0,\infty)}{1\over(\alpha+\lambda)^m}\,d\mu(\lambda)\in[0,\infty)
\end{equation}
so that
\begin{equation}\label{eq-5.9}
\int_{[0,\infty)}{1\over(\alpha+\lambda)^m}\,d\mu(\lambda)<+\infty.
\end{equation}
Since
$$
{(x-\alpha)^m\over(x+\lambda)(\alpha+\lambda)^m}
={(x-\alpha)^{m-1}\over(\alpha+\lambda)^m}
-{(x-\alpha)^{m-1}\over(x+\lambda)(\alpha+\lambda)^{m-1}},
$$
one can write
\begin{align*}
f(x)&=\sum_{l=0}^{m-2}{f^{(l)}(\alpha)\over l!}\,(x-\alpha)^l+\gamma_1(x-\alpha)^{m-1}
+\int_{[0,\infty)}{(x-\alpha)^{m-1}\over(x+\lambda)(\alpha+\lambda)^{m-1}}
\,d\mu(\lambda).
\end{align*}
Thanks to \eqref{eq-5.8} and \eqref{eq-5.9} this yields that
$f\in\cP^{(m-1)}(0,\infty)$.
\end{proof}

By using Proposition \ref{P-5.2} and Lemma \ref{L-5.6}, it is not difficult to
prove the following chacterizations of operator absolutely/completely monotone
functions on $(0,\infty)$ (introduced in the previous section).

\begin{prop}
Let $f$ be a smooth real function on $(0,\infty)$. Then $f$ is operator absolutely monotone
if and only if $f(x)=\alpha x+\beta$ with $\alpha,\beta\ge0$. Also, $f$ is operator
completely monotone if and only if $f$ is a non-negative operator monotone decreasing
function on $(0,\infty)$.
\end{prop}

\section{Examples}\label{sec-6}

In this section we present several examples of operator $k$-tone functions on $(0,\infty)$.
First, recall a convenient way to obtain operator $k$-tone functions on $(0,\infty)$. Let
$g$ be an operator monotone function on $(0,\infty)$, $k,m\in\bN$ with $m<k$, and let
$\alpha_1,\dots,\alpha_m\in(0,\infty)$. By Corollary \ref{C-3.4} and
Proposition \ref{P-5.2} we obtain
\begin{equation}\label{eq-6.1}
(-1)^{k-m-1}\Biggl\{\prod_{l=1}^m(x-\alpha_l)\Biggr\}g(x)\in\cP^{(k)}(0,\infty).
\end{equation}
Furthermore, by taking the limit as $\alpha_l\searrow0$, \eqref{eq-6.1} is valid for any
$\alpha_1,\dots,\alpha_m\in[0,\infty)$ with any $m<k$.

\begin{example}\rm
Note that $-x^r$ and $x^r(x-\alpha)$ are operator monotone on $(0,\infty)$ for every
$r\in[-1,0]$ and $\alpha\in[0,\infty)$. For each $k\in\bN$ and every
$\alpha_1,\dots,\alpha_k\in[0,\infty)$, by \eqref{eq-6.1} we have
$$
(-1)^{k-m}x^r\prod_{l=1}^m(x-\alpha_l)\in\cP^{(k)}(0,\infty)
$$
if $r\in[-1,0]$ and $m=0,1,\dots,k$.

Concerning the power functions $\pm x^p$ on $(0,\infty)$ with $p\in\bR$, let us prove
that, for each $k\in\bN$, $x^p\in\cP^{(k)}(0,\infty)$ if and only if
\begin{equation}\label{eq-6.2}
\begin{cases}
p\in[0,1]\cup[2,3]\cup\dots\cup[k-1,k] & \text{if $k$ is odd}, \\
p\in[-1,0]\cup[1,2]\cup\dots\cup[k-1,k] & \text{if $k$ is even},
\end{cases}
\end{equation}
and $-x^p\in\cP^{(k)}(0,\infty)$ if and only if
\begin{equation}\label{eq-6.3}
\begin{cases}
p\in[-1,0]\cup[1,2]\cup\dots\cup[k-2,k-1] & \text{if $k$ is odd}, \\
p\in[0,1]\cup[2,3]\cup\dots\cup[k-2,k-1] & \text{if $k$ is even}.
\end{cases}
\end{equation}
We need to prove the ``only if" parts.
By Lemma \ref{L-5.6} note that $p\in[-1,k]$ is
necessary for $\pm x^p$ to belong to $\cP^{(k)}(0,\infty)$. Since
$$
{d^k\over dx^k}\,x^p=p(p-1)\cdots(p-k+1)x^{p-k},\qquad x\in(0,\infty),
$$
the condition $p(p-1)\cdots(p-k+1)\ge0$ (resp., $p(p-1)\cdots(p-k+1)\le0$) is necessary
for $x^p$ (resp., $-x^p)$ to belong to $\cP^{(k)}(0,\infty)$.
Therefore, \eqref{eq-6.2} is
necessary for $x^p$ to belong to $\cP^{(k)}(0,\infty)$, and \eqref{eq-6.3} is necessary for
$-x^p$ to belong to $\cP^{(k)}(0,\infty)$.
\end{example}

\begin{example}\rm
Since $\log x$ is operator monotone on $(0,\infty)$, for each $k\in\bN$ and every
$\alpha_1,\dots,\alpha_{k-1}\in[0,\infty)$, by \eqref{eq-6.1} we have
$$
(-1)^{k-m-1}\Biggl\{\prod_{l=1}^m(x-\alpha_i)\Biggr\}\log x\in\cP^{(k)}(0,\infty),
\qquad m=0,1,\dots,k-1.
$$

Concerning the functions $\pm x^p\log x$ on $(0,\infty)$ with $p\in\bR$, we prove that, for
each $k\in\bN$, $x^p\log x\in\cP^{(k)}(0,\infty)$ if and only if
$$
\begin{cases}
p\in\{0,2,\dots,k-1\} & \text{if $k$ is odd}, \\
p\in\{1,3,\dots,k-1\} & \text{if $k$ is even},
\end{cases}
$$
and $-x^p\log x\in\cP^{(k)}(0,\infty)$ if and only if
$$
\begin{cases}
p\in\{1,3,\dots,k-2\} & \text{if $k$ is odd (empty if $k=1$)}, \\
p\in\{0,2,\dots,k-2\} & \text{if $k$ is even}.
\end{cases}
$$
It is enough to prove the ``only if" parts. By induction one can compute
$$
{d^k\over dx^k}\bigl(x^p\log x\bigr)=x^{p-k}\{p(p-1)\cdots(p-k+1)\log x+Q_k(p)\},
\qquad k=1,2,\dots,
$$
where $Q_1(p):=1$ and
$$
Q_k(p):=\sum_{i=0}^{k-1}p(p-1)\cdots(p-i-1)(p-i+1)\cdots(p-k+1),\qquad k\ge2.
$$
Hence, if $p\not\in\{0,1,\dots,k-1\}$, then ${d^k\over dx^k}\bigl(x^p\log x\bigr)$ takes
both positive and negative values on $(0,\infty)$, so neither $x^p\log x$ nor $-x^p\log x$
belongs to $\cP^{(k)}(0,\infty)$. Moreover, if both $\pm x^p\log x$ belong to
$\cP^{(k)}(0,\infty)$, then $x^p\log x$ must be a polynomial of at most degree $k-1$, which
is impossible. Combining these facts shows the assertions.
\end{example}

\begin{example}\rm
Note that $-1/(x+1)$ and $x/(x+1)$ are operator monotone on $(0,\infty)$. By \eqref{eq-6.1},
for each $k\in\bN$ and every $\alpha_1,\dots,\alpha_k\in[0,\infty)$,
$$
{(-1)^{k-m}\over x+1}\,\prod_{l=1}^m(x-\alpha_l)\in\cP^{(k)}(0,\infty),
\qquad m=0,1,\dots,k.
$$

Concerning the functions $\pm x^p/(x+1)$ on $(0,\infty)$ with $p\in\bR$, let us show that,
for each $k\in\bN$, $x^p/(x+1)\in\cP^{(k)}(0,\infty)$ if and only if
$$
\begin{cases}
p\in\{1,3,\dots,k\} & \text{if $k$ is odd}, \\
p\in\{0,2,\dots,k\} & \text{if $k$ is even},
\end{cases}
$$
and $-x^p/(x+1)\in\cP^{(k)}(0,\infty)$ if and only if
$$
\begin{cases}
p\in\{0,2,\dots,k-1\} & \text{if $k$ is odd}, \\
p\in\{1,3,\dots,k-1\} & \text{if $k$ is even}.
\end{cases}
$$
By induction one can compute the $k$th derivative
$$
{d^k\over dx^k}\bigl(x^p/(x+1)\bigr)=x^{p-k}(x+1)^{-(k+1)}Q_k(x),
\qquad k=1,2,\dots,
$$
where $Q_k(x)$ is a polynomial given as
\begin{align*}
Q_k(x)&=(p-1)(p-2)\cdots(p-k)x^k \\
&\qquad+\alpha_{k-1}^{(k)}x^{k-1}+\dots+\alpha_1^{(k)}x+p(p-1)\cdots(p-k+1),
\end{align*}
that is, $Q_k(x)$ is a polynomial of at most degree $k$ with the coefficient
$(p-1)(p-2)\cdots(p-k)$ of $x^k$ and the constant term $p(p-1)\cdots(p-k+1)$.
If $x^p/(x+1)\in\cP^{(k)}(0,\infty)$, then $Q_k(x)\ge0$ for all  $x\in(0,\infty)$ and
we must have
$$
(p-1)(p-2)\cdots(p-k)\ge0,\qquad p(p-1)\cdots(p-k+1)\ge0.
$$
These imply that
$$
\begin{cases}
p\in\{1,2,3,\dots,k-1\}\cup[k,\infty) & \text{if $k$ is odd}, \\
p\in(-\infty,0]\cup\{1,2,3,\dots,k-1\}\cup[k,\infty) & \text{if $k$ is even}.
\end{cases}
$$
On the other hand, if $-x^p/(x+1)\in\cP^{(k)}(0,\infty)$, then we must similarly have
$$
\begin{cases}
p\in(-\infty,0]\cup\{1,2,3,\dots,k-1\} & \text{if $k$ is odd}, \\
p\in\{1,2,3,\dots,k-1\} & \text{if $k$ is even}.
\end{cases}
$$
Here $\pm x^p/(x+1)$ cannot belong to $\cP^{(k)}(0,\infty)$ at the same time. Hence
it remains to show that $x^p/(x+1)$ does not belong to $\cP^{(k)}(0,\infty)$ if
$p\in(k,\infty)$ and that $\pm x^p/(x+1)$ do not belong to $\in\cP^{(k)}(0,\infty)$ if
$p\in(-\infty,0)$. In the case $k=1$, these can be shown by appealing to the analytic
continuation property (as Pick functions) of operator monotone functions. Then one can use
an induction argument based on a characterization of $f\in\cP^{(k)}(0,\infty)$ in terms of
$f^{[k-1]}(x,\alpha,\dots,\alpha)$ given in Theorem \ref{T-3.3} while the details are
omitted.
\end{example}

\begin{example}\rm
It is well known that $(x-1)/\log x$ is operator monotone on $(0,\infty)$. By
\eqref{eq-6.1}, for each $k\in\bN$ and every $\alpha_1,\dots,\alpha_{k-1}\in[0,\infty)$,
$$
(-1)^{k-m-1}\Biggl\{\prod_{l=1}^m(x-\alpha_l)\Biggr\}\,{x-1\over\log x}
\in\cP^{(k)}(0,\infty),\qquad m=0,1,\dots,k-1.
$$

Concerning the functions $\pm x^p(1-x)/\log x$ on $(0,\infty)$ with $p\in\bR$, we notice
that, for each $k\in\bN$, $x^p(1-x)/\log x\in\cP^{(k)}(0,\infty)$ if and only if
$$
\begin{cases}
p\in\{0,2,\dots,k-1\} & \text{if $k$ is odd}, \\
p\in\{1,3,\dots,k-1\} & \text{if $k$ is even},
\end{cases}
$$
and $-x^p/\log x\in\cP^{(k)}(0,\infty)$ if and only if
$$
\begin{cases}
p\in\{1,3,\dots,k-2\} & \text{if $k$ is odd (empty if $k=1$)}, \\
p\in\{0,2,\dots,k-2\} & \text{if $k$ is even}.
\end{cases}
$$
The proof of these assertions is similar to those of the above examples.
We omit the details.
\end{example}

\section*{Acknowledgments}

U.F.\ and E.R. gratefully acknowledge support by the ANR project OSQPI (ANR-11-BS01-0008).
The work of F.H. was partially supported by Grant-in-Aid for Scientific Research
(C)21540208. This work was carried out when F.H.\ was a CNRS research fellow
in 2010-2011 at the Dept.\ of Mathematics of Universit\'e de Franche-Comt\'e.

\end{document}